\documentclass[12pt,a4paper]{scrartcl}
\usepackage{amsmath}\usepackage{amssymb}
\usepackage{amsthm}
\usepackage[english]{babel}
\usepackage{nicefrac}  \usepackage{floatflt}
\usepackage[pdftex]{graphicx}    
\usepackage[sortcites=true,maxbibnames=99,giveninits=true,doi=false,url=false,isbn=false]{biblatex}
\newcommand{\R}{\mathbb{R}}

\newcommand{\Z}{\mathbb{Z}}

\newcommand{\A}{\mathcal{A}}

\newcommand{\vol}{\mathrm{vol}}
\newcommand{\Span}{\mathrm{span}}
\newcommand{\lsim}{\lesssim}
\newcommand{\gsim}{\gtrsim}
\newcommand{\conv}{\mathrm{conv}}

\newcommand{\N}{\mathbb{N}}
\newcommand{\mo}{\; \mathrm{mod} \;}
\newcommand{\eps}{\varepsilon}

\title{A duality principle for the neighbourhood of manifolds and applications to counting rational points near curves}
\author{ Lauritz Streck}
\date{\today}

\begin{document}
\newtheorem{theorem}{Theorem}
\newtheorem{definition}[theorem]{Definition}
\newtheorem{question}[theorem]{Question}
\newtheorem{remark}[theorem]{Remark}
\newtheorem{lemma}[theorem]{Lemma}
\newtheorem{corollary}[theorem]{Corollary}
\newtheorem{prop}[theorem]{Proposition}
\newtheorem{claim}[theorem]{Claim}
\newtheorem{conj}[theorem]{Conjecture}
\newcommand{\restatedlabel}{}
\newtheorem*{restatedinner}{\restatedlabel}
\newenvironment{restated}[1]{\renewcommand{\restatedlabel}{#1}\begin{restatedinner}}{\end{restatedinner}}

\maketitle
\begin{abstract}
In this paper,  we establish a duality principle relating the count of rationals near a manifold to a sparse counting problem for the dual.  As applications,  we give a short,  self-contained proof of sharp bounds for planar curves under minimal regularity conditions and similarly,  sharp upper bounds for non-degenerate space curves,  again under minimal curvature conditions.  This work recovers and extends the works of Beresnevich,  Dickinson and Velani and of Vaughan and Velani on planar curves and independently proves and extends the results of Chen,  Seeger,  Srivastava and Technau on space curves,  while providing a structural picture well-suited for generalisations.  The results and methods of this work precede the emergence of highly capable AI models.
\end{abstract}

\section{Introduction}
In this paper,  we establish a duality principle and detail applications of it to the question of how many rational points up to a given height lie in the small neighbourhood of a bounded manifold $\mathcal{M}$.  In this first section,  we will briefly state the problem and trace some of its history before we sketch our duality principle and the results coming out of it.  In the next sections of the introduction,  we will then give a more thorough overview of the previous papers,  state our duality principle and detail its applications to planar and space curves.  There,  we will also give a short, direct and self-contained proof of sharp bounds for planar curves,  which despite its simplicity gives much better bounds under weaker regularity assumptions than all currently known arguments.  

At the end of the introduction is the AI disclaimer (used only in the final phase of this project; all significant parts of the paper precede the current highly powerful AI models) and some reflections on how this paper fits into the emerging mathematical world dominated by the use of AI models.  

\subsection{Problem statement}
To state the counting problem of rational points,  fix a manifold $\mathcal{M}$,  an integer $Q$ and an $\eps>0$.  Let 
\[
\mathcal{A}_Q:=\{\vec{a}/q: (\vec{a}, q) \in \Z^{n+1}, Q/2 \leq q \leq Q \}
\]
denote the rational points up to height $Q$.  Letting $\eps:=Q^{-(1+\tau)}$ decrease with $Q$,  we are then interested in the count and the structure of the set $\mathcal{A}_Q \cap B_\eps(\mathcal{M})$,  where $B_\eps(\mathcal{M})$ is the $\eps$-neighbourhood of the manifold $\mathcal{M}$.  To be consistent with previous notation,  we set $\delta=\eps Q$ and $\mathcal{N}(\delta, Q)=\#(\mathcal{A}_Q \cap B_\eps(\mathcal{M}))$.  

For a unified behaviour over a class of manifolds of a given dimension,  one has to rule out the manifold being close to a lower dimensional subspace (as the expected behaviour has to depend on the dimension).  This forces us to impose curvature conditions on $\mathcal{M}$; in the hypersurface case, the minimal possible one is that of $\mathcal{M}$ not having planar points and in the curve case,  it is that of the curve being non-degenerate.  

In terms of $\tau$,  for a given manifold of dimension $k$ and co-dimension $n-k$,  the behaviour of the set $\mathcal{A}_Q \cap B_\eps(\mathcal{M})$ exhibits two phase transitions: For small $\tau$,  the total number of rational points in $B_\eps(\mathcal{M})$ up to height $Q$ is controlled by the bound expected from the volume $\eps^{n-k} Q^{n+1}$.  For intermediate $\tau$,  the total number of rationals might exceed the volume bound,  but only because of the contributions of small patches of the manifold aligned with lower dimensional subspaces.  In this range,  one can still ask about the structure of the set $\mathcal{A}_Q \cap B_\eps(\mathcal{M})$,  for example in terms of their Hausdorff dimension.  Finally,  for large $\tau$ the structure breaks down,  for example in that the counts obtained cease to agree when translating the manifold in $\R^n$.  

For planar curves and other hypersurfaces,  the two phase transitions agree,  but for more general manifolds,  they are typically distinct.

\subsection{Sketch of methods}
Let $\mathcal{M}$ be a manifold and $\eps>0$ be a parameter; we now sketch our main duality result.  Assume that the manifold is explicitly parametrised by $\mathcal{M}=\left\{ (x, f(x)): x \in U \right\}$ for $f : \R^k \to \R^{n-k}$.  Tile the domain into patches $U_j$ of diameter $\sqrt{\eps}$ with center points $x_j$ and linearise;  by Taylor approximation,  the family of boxes 
\[
V_j=\{(y_1, y_2) \in \R^k \times \R^{n-k}: y_1 \in U_j,  |y_2-f(x_j)-Df(x_j) \cdot (y_1-x_j)| \leq \eps \},
\]
with $Df \in Mat_{n-k,  k}$ the derivative of $f$,  then tiles $B_\eps(\mathcal{M})$ up to small error.  Now,  assume that $\eps$ and $Q$ are such that $\eps^{k/2} \eps^{n-k} \geq Q^{-(n+1)}$,  which means that a random translation of a patch would be expected to contain some rational point of height up to $Q$. Our main duality result then essentially states that for any given box $V_j$ that contains significantly more or significantly fewer rationals up to height $Q$ than expected,  there must be rational points of height up to some $\widetilde{Q}$ in the dual box $\widetilde{V}_j$ (which arises from projective duality and is defined below).  The patching together of the dual of all of the boxes then leads to what we call the dual counting problem,  which for manifolds with the minimal curvature assumptions needed corresponds locally to the linearised covering of another manifold.  For example,  for hypersurfaces of positive Gaussian curvature,  this new manifold is the Legendre dual,  while for a non-degenerate curve in $\R^n$ it is the projective dual of the curve corresponding to the osculating hyperplanes.  However,  the result also allows us to dualise manifolds with much less rigid curvature assumptions locally.  

This leaves us with a new counting problem, but the $\widetilde{V}_j$ and $\widetilde{Q}$ are this time such that $\vol(\widetilde{V}_j)<\widetilde{Q}^{-(n+1)}$ so we are in what can be described as the `sparse' setting,  in that we do not expect the patches to contain a rational point.  Our second insight is that for manifolds with curvature,  the number of patches containing such a rational can be controlled purely geometrically using the determinant method,  which implies that any $n+1$ many rational points of height up to $\widetilde{Q}$ not contained in an affine hyperplane must span a parallelepiped of volume at least $\widetilde{Q}^{-(n+1)}$.  Combining the two ideas,  we get our results.  

We end with a broader take on our methods. Namely,  there are some reasons to expect a version of the duality in this paper to be useful much more widely than for the specific problems considered in this paper.  In many seemingly geometric problems like the ones considered in the current paper,  Fourier analysis is used as a tool.  I have always found the appearance of Fourier analysis in this context somewhat puzzling: Why would one expect the decomposition into oscillations to be necessary for a geometric problem? 

In the context of the problem of estimating rationals near curves,  the gist of the answer is thus: Introducing oscillations is neither necessary nor useful,  but duality through Fourier analysis still gives something because it also captures the decomposition into directions and hyperplanes.  However,  in the context of the problems in this paper,  it merely acts as a crutch for the more direct projective duality principle developed,  and switching to it leads to a more direct,  intuitive and simpler approach.  It seems quite possible that in many other `geometric looking problems' like the currently much considered incidence problems,  using the direct duality principle in Lemma \ref{lem::boxdual} and Theorem \ref{thm::nearcurvedual} instead of Fourier analysis could lead to a more direct and powerful approach,  and that many of the decompositions and methods used might be translatable to a geometrically intuitive decomposition of the points near the dual object.  If so,  this approach might lead to valuable new perspectives.     

\subsection{Previous progress on the problem}
The problem of counting rational points near curves and manifolds is a very active research area; instead of citing all results,  we focus on the ones most relevant to us and only cite the recent work of Beresnevich and Datta \cite{beresnevich-datta,  beresnevich-datta-2} as an example of obtaining partial results for rationals near manifolds with the aid of the Kleinbock-Margulis result \cite{kleinbock-margulis} from homogeneous dynamics.  For further results on rational points near manifolds,  see for example \cite{beresnevich,  beresnevich-vaughan-velani-zorin-conv,  beresnevich-vaughan-velani-zorin-div,  huang,  schindler-srivastava-technau,  beresnevich-yang} and the references therein.  

The first progress on the problem was made for planar curves by Huxley \cite{huxley},  who showed that $\mathcal{N}(\delta, Q) \lsim \delta^{1-\eps} Q^2$ for any $\eps>0$.  Huxley used the determinant method and the Legendre dual of the curve together with direct computation and an iteration scheme to obtain his result.  Subsequently,  the papers \cite{beresnevich-dickinson-velani} and \cite{vaughan-velani} together established the sharp upper and lower bounds $\delta Q^2$ with error terms of $Q^{1+\eps}$ for the count, and Huang then established an asymptotic formula \cite{huangasymp}.  Both \cite{huangasymp} and \cite{vaughan-velani} relied on Fourier analysis,  had an error term of $Q^{1+\eps}$ and additional regularity assumptions.  Moreover, the two papers are very indirect in that they rely on Huxley's bound in some form.  

Subsequently,  Huang showed an estimate for the rational points near a hypersurface of positive Gaussian curvature \cite{huang}.  He also used Fourier analysis to induce duality,  which left him with a new counting problem near the Legendre dual in which the parameters only went up to `sparse' height (we will make this more precise below).  His insight was that while he could not treat the sparse case directly,  he could artificially enlarge the neighbourhood to get back to the dense setting and iterate the result there.  Because the numerics are so favorable in larger dimensions for manifolds of positive curvature,  this allowed him to arrive at the sharp threshold in dimensions $ n \geq 3$.  This approach is one of the few results not using Huxley's bound,  but the rather indirect approach also leads to the relatively rigid assumption of positive Gaussian curvature and a relatively unwieldy error term in dimension $n=3$.  For planar curves,  this approach does not reach the sharp threshold.

In very recent work,  Chen,  Seeger,  Srivastava and Technau \cite{chen-seeger-srivastava-technau} showed that for space curves,  the total count of rational points is at most the expected count of $\delta^2 Q^2$ as long as $\delta \geq Q^{-1/3}$.  This range of $\delta$ is sharp because of the possible presence of a single line segment containing $\eps^{1/2} Q^2$ many rational points.  Their proof relies on Fourier analysis to induce duality and through the result of \cite{vaughan-velani} also on Huxley's result for planar curves.  

In a separate approach,  Badziahin \cite{badziahin-veronese,  badziahin-3d-veronese} and Badziahin and Shulga \cite{badziahin-shulga} studied the Hausdorff dimension structure of the moment curve in $\R^3$.  The first two papers reached the value of $\tau=3/5$ conjectured by Beresnevich and Yang \cite{beresnevich-yang},  while the last one moves beyond it.  For space curves,  the question of Hausdorff dimension of well-approximable numbers makes sense beyond the threshold at which the total count breaks down due to the presence of the short line segments alluded to above,  and their work is strictly stronger than what applying the work of Chen, Seeger, Srivastava and Technau \cite{chen-seeger-srivastava-technau} to the moment curve gives.  The work of Badziahin and Badziahin--Shulga relies on bounds for the count of cubic forms of a given discriminant by Davenport \cite{davenport-1,  davenport-2} and Hooley \cite{hooley}. 

Finally,  the recent AI-assisted paper of Gan,  Guo and Oh \cite{gan-guo-oh} gives a sharp counting result for curves in $\R^n$,  $n \geq 4$,  with the use of Fourier analysis and the wave packet method of Gan,  Maldague and Oh \cite{gan-maldague-oh}.  They note that their results explicitly do not imply bounds for the case of planar curves.
\subsection{Results}
One of the two main ideas of this paper is to use projective duality directly,  rather than Fourier analysis to induce the duality.  Our first result is the following. 

Here $\mathcal{N}(Q, V)$ denotes the number of rational points of height at most $Q$ in $V$,  and the dual vectors $\widetilde{\Phi e_j}$ and dual boxes are defined in Section \ref{subsec::boxtransform}.
\begin{theorem} \label{thm::maindualitybox}
Let $n$ be an integer.  There are constants $c_0, C>0$,  an integer $N$ and integer matrices $\Phi_1, \dots, \Phi_N \in SL_n(\Z)$ with orthogonal columns, all depending only on the dimension (picking $c_0$ first,  then $N$,  and then $C$),  such that the following holds.  For any $x_0 \in \R^n$,  for any orthonormal basis $e_1,  \dots, e_n$ and for any $c_0 \geq \eps_1 \geq \dots \geq \eps_n >0$ and $Q \in \mathbb{N}$ such that 
\[
\eps_1 \cdot \dots \cdot \eps_n \geq C Q^{-(n+1)},
\]
consider the box 
\[
V:=\{x_0+\sum_{j=1}^n \beta_j e_j,  |\beta_j| \leq \eps_j \forall 1 \leq j \leq n \}.
\]
Set $ \widetilde{\eps_j}:=\frac{\eps_n}{\eps_{n-j}}$ with $\eps_0:=1$ for $1 \leq j \leq n$ and $\widetilde{Q}:=\frac{1}{\eps_n Q}$.

For any parameter $A \geq 1$,  if
\[
\mathcal{N}(Q,  V) \geq A \eps_1 \dots \eps_n Q^{n+1},
\]
the rational points in $V$ up to height $Q$ must then lie in a rational subspace of dimension $n-r$ for some $r \geq 1$.  There is then a $\Phi=\Phi_j$ from the above finite collection with $(\Phi e_n)_n \geq c_0$ such that for $V_\Phi:=x_0+\Phi (V-x_0)$, the dual box
\[
\widetilde{V}_\Phi=\left\{x_{V_\Phi}+\sum_{j=1}^n \beta_j \widetilde{\Phi e_j},  |\beta_j| \leq \widetilde{\eps_j}\right\},
\]
where $x_{V_\Phi}$ is the dual point of the hyperplane orthogonal to $\Phi e_n$ passing through $x_0$,  satisfies the following.

For $1 \leq i \leq r$, there are $B_i \geq 1$, $m_i \leq \widetilde{Q}/B_i$ such that
\[
\prod_{i=1}^r \frac{\widetilde{Q}}{B_i m_i} \gsim A
\]
and there are rational points $\vec{a_i}/m_i \in (\widetilde{V}_\Phi)_{B_i} \cap K(C)$,  where $\vec{a_i} \in \Z^n$ and $K(C)$ is the compact ball around the origin with radius $C$. The rational points $\vec{a_i}/m_i$ are in general position,  meaning that no $s+2$ of them are contained in an $s$-dimensional affine subspace.  The rational subspace of dimension $n-r$ above is $H_1 \cap \dots \cap H_r$,  where $H_i$ is the hyperplane corresponding to $\vec{a_i}/m_i$ under projective duality.

Moreover,  if the box is such that $\vec{a}/m \notin (\widetilde{V}_\Phi)_{C \widetilde{Q}/m} \cap K(C)$ for all $\Phi=\Phi_j$ and all integers $(m,  \vec{a}) \in \Z^{n+1}$,  then 
\[
\mathcal{N}(Q,  V) \gsim \eps_1 \dots \eps_n Q^{n+1}.
\]
\end{theorem}

The finitely many matrices $\Phi_j$ only fix a system of coordinates in which projective duality is well behaved for the box at hand; their role is explained in Remark \ref{rem::Phiexplanation}.

For a given curve,  we can linearise its $\eps$-neighbourhood into boxes and apply the duality to each box individually.  The dual box of each linearised box corresponds to the box around the dual curve,  which leads to the following.  

In the statement below,  $B_{\mathbf \eps}(t)$ denotes the $\eps_1 \times \dots \times \eps_n$-box adapted to the curve at time $t$,  $\Delta$ the associated step size and $V(\mathbf \eps)$ the adapted neighbourhood; these and the minimal cover are defined in Section \ref{sec::dualcurves}.
\begin{theorem} \label{thm::mainduality}
There is a finite collection of matrices $\Phi_k \in SL_n(\Z)$ with orthogonal columns such that the following holds. Let $\gamma$ be a non-degenerate curve in $\R^n$ with dual curve $\gamma^\ast$,  let $c_0 \geq \eps_1 \geq \dots \geq \eps_n>0$ and let $Q \in \N$ be such that
\[
\eps_1 \cdot \dots \cdot \eps_n \geq C Q^{-(n+1)},
\]
where $c_0$ and $C$ are the constants from Theorem \ref{thm::maindualitybox}.  Let $\Delta$ be the associated step size,  $N:=\lfloor \Delta^{-1} \rfloor$,  and let $t_1,  \dots,  t_N$ be as in Lemma \ref{lem::minimalcover}.  Let $V_j:=B_{\mathbf \eps}(t_j)$,  $1 \leq j \leq N$,  be the boxes of the minimal cover of $V(\mathbf \eps)$,  and let $\widetilde{V}_j$ be the $\widetilde{\eps_1} \times \dots \times \widetilde{\eps_n}$-box adapted to the dual curve $\gamma^\ast$ at time $t_j$,  where $\widetilde{\eps_i}=\eps_n/\eps_{n-i}$.  For $B \geq 1$,  let $(\widetilde{V}_j)_B$ denote the corresponding box with parameters $B\widetilde{\eps_1},  \dots,  B\widetilde{\eps_n}$.  Set $\widetilde{Q}:=\frac{1}{\eps_n Q}$.  

For any parameter $A \geq 1$,  if
\[
\mathcal{N}(Q,  V_j) \geq C A \eps_1 \dots \eps_n Q^{n+1}
\]
for some $j$, then there is a $\Phi_k$ in the collection above such that for all $1 \leq i \leq r<n$ there are $B_i \geq 1$ and $m_i \leq \widetilde{Q}/B_i$ such that
\[
\prod_{i=1}^r \frac{\widetilde{Q}}{B_i m_i} \gsim A,
\]
and there are rational points $\vec{a_i}/m_i \in (\widetilde{\Phi_k V_j})_{B_i} \cap K(C)$ with $\vec{a_i} \in \Z^n$.  The rational points $\vec{a_i}/m_i$ are in general position,  meaning that no $s+2$ of them are contained in an $s$-dimensional affine subspace.  The rational points in $V_j$ lie in the $(n-r)$-dimensional rational subspace $H_1 \cap \dots \cap H_r$,  where $H_i$ is the hyperplane corresponding to $\vec{a_i}/m_i$ under projective duality (in the coordinates given by $\Phi_k$).  Here, the $\widetilde{\Phi_k V_j}$ are the boxes of the adapted cover of the dual curve $(\Phi_k \gamma)^\ast$.
\end{theorem}

This gives the satisfying description that the only reason for an excess of rational points near the curve must come from rational points of small denominator near (one of the versions of) the dual curve.  Moreover,  because 
\[
\widetilde{Q}^{-(n+1)}=\eps_n^{n+1}Q^{n+1} \geq C\frac{\eps_n^{n+1}}{\eps_1 \dots \eps_n}=C\,\widetilde{\eps_1} \dots \widetilde{\eps_n},
\]
the counting problem near the dual curve is in the sparse setting,  meaning that we on average expect only a small proportion of the boxes to contain one of the problem rational points.  

One could of course do something similar for a manifold $\mathcal{M}=\{(x,  f(x)): x \in U \subset \R^k\}$,  decomposing the manifold into cubes of diameter $\sqrt{\eps}$.  One advantage of this is that we can dualise locally,  or only along a specific arc of the manifold like an arc of significant curvature.  This makes dualising manifolds with very weak curvature assumptions feasible. For example,  for surfaces in $\R^3$, this seems to open up a clear path to extending Huang's result \cite{huang} to surfaces with weaker curvature conditions. Indeed, as long as we are far from a planar point, there is a direction along which we can decompose $U$ into squares $U_{i, j}$ of diameter $\sqrt{\eps}$ such that the patches $V_{i, j}:=\{(x, y): x \in U_{i, j},\ |y-f(x)| \leq \eps\}$ in each row with $j=j_0$ follow a non-degenerate planar curve, thickened by $\sqrt{\eps}$ in the orthogonal direction. Applying the determinant method, we know that the `generic' points in $\{V_{i, j}, j=j_0\}$ cannot lead to an excess of rationals. The case of planes and lines intersecting the $V_{i, j}$ leads to a case distinction, almost all of which cases I was able to treat with volume arguments or the planar case when I had to cut the exploration of this project short.

Now,  the second idea comes into play,  which is that instead of dualising again,  we use the determinant method directly,  in the following form.

\begin{lemma} \label{lem::ratrepellancy}
Let $\vec{p}_1, \dots, \vec{p}_{n+1} \in \Z^n$ and $q_1, \dots, q_{n+1} \in \Z_{>0}$ such that $q_i \leq Q$ for some $Q$.  Define $v_i:=\vec{p}_i/q_i \in \R^n$.  Then
\[
\vol(\conv(v_1, \dots, v_{n+1})) \gsim Q^{-(n+1)}
\]
unless there is an affine hyperplane containing all $v_1, \dots, v_{n+1}$. 

Moreover,  if $H:=\{x: k_0+k_1 x_1+\dots+k_{n-1} x_{n-1}+m x_n=0\}$ is a rational hyperplane where $k_j$ and $m$ are integers with $\gcd(m, k_0, \dots, k_{n-1})=1$ and $m>0$, then for any collection of rational vectors $(\vec{a}_1/q_1,  b_1/q_1), \dots,  (\vec{a}_n/q_n,  b_n/q_n)$ contained in $H$,  where $\vec{a}_i \in \Z^{n-1}$ and again $q_i \leq Q$,
\[
\vol_{n-1}(\conv(\vec{a}_1/q_1, \dots, \vec{a}_n/q_n)) \gsim \frac{m}{Q^{n}}
\]
unless there is a subspace of dimension $n-2$ containing all of the rational points.  The implicit constants depend only on the dimension.
\end{lemma}

The combination of the two ideas is powerful enough that we immediately obtain the following result for planar curves,  which despite its short and direct proof significantly improves on all of the three papers \cite{huangasymp,  beresnevich-dickinson-velani,  vaughan-velani}.  

To state it,  fix an interval $I \subset \R$ and a twice continuously differentiable function $f \colon I \to \R$ with $f^{\prime \prime}(x) \geq c>0$ for all $x \in I$,  and for any sub-interval $J \subset I$ let
\[
\mathcal{N}(\delta, Q, J):=\#\left\{(a, b, q) \in \Z^3,  \frac{1}{2} Q \leq q \leq Q: \; \frac{a}{q} \in J,  \; \left|f\left(\frac{a}{q}\right)-\frac{b}{q}\right| \leq \frac{\delta}{Q} \right\}
\]
count the rational points near the curve segment $(x, f(x))$,  $x \in J$,  and write $\mathcal{N}(\delta, Q):=\mathcal{N}(\delta, Q, I)$.
\begin{theorem} \label{thm::R2curves}
There is a constant $C$,  depending only on $f$,  such that for all $Q \geq 2$ and all $\delta \geq C Q^{-1}(\log Q)^2$,
\[
\mathcal{N}(\delta, Q)=\frac{7}{12}\,\delta Q^2 |I| \left(1+O\left(\frac{(\log Q)^{1/2}}{(\delta Q)^{1/4}}\right)\right).
\]
Moreover,  $\mathcal{N}(\delta, Q) \lsim \delta Q^2+Q(\log Q)^2$ for all $\delta>0$.  The implicit constants depend only on $f$.
\end{theorem}

\begin{proof}[Sketch of Proof]
Write $\eps:=\delta/Q$.  In this sketch,  we only consider lines whose intersection with the $\eps$-neighbourhood of the curve has the maximal possible length $\sqrt{\eps}$; lines at a bigger angle to the curve can be treated similarly.  After applying Theorem \ref{thm::maindualitybox},  the goal is to bound the number of boxes $\widetilde{V}_i$ of length $\sqrt{\eps}$ around the dual curve containing a rational of height $m$,  with $m \lsim \delta^{-1}$.   Let $c_0$ be a small number and $T$ be the smallest integer such that for $T$ many adjacent boxes $\widetilde{V}_{i_0},  \dots,  \widetilde{V}_{i_0+T}$,
\[
\vol(\conv(\widetilde{V}_{i_0},  \dots,  \widetilde{V}_{i_0+T}))=(T \sqrt{\eps})^3 \geq c_0 m^{-3},
\]
meaning that 
\[
T \gsim \frac{1}{\sqrt{\eps} m}.
\]
By the determinant method (Lemma \ref{lem::ratrepellancy}),  all of the rationals up to height $m$ contained within $\conv(\widetilde{V}_{i_0},  \dots,  \widetilde{V}_{i_0+T})$ lie on a line, but any line can intersect at most $O(1)$ many boxes because of the curvature.  Hence,  there is at most a 
\[
\frac{1}{T} \lsim \sqrt{\eps} m \lsim (\delta Q)^{-1/2}
\] 
proportion of the boxes that contain a rational,  and the total contribution of all of the lines corresponding to these rationals is at most
\[
\frac{1}{T} \frac{Q^2}{m} \lsim \sqrt{\eps} Q^2 \lsim \delta Q^2 (\delta Q)^{-1/2}.
\]
Picking $\delta$ big enough in $Q$ to beat the log terms from the dyadic partitioning shows the claim. 
\end{proof}

Finally,  we present our results in the case of curves in $\R^3$.  As mentioned above,  the result below was independently obtained in \cite{chen-seeger-srivastava-technau} under stronger curvature assumptions and with a significantly less direct proof.  For a non-degenerate curve $\gamma : [0, 1] \to \R^3$,  let $\mathcal{N}(\delta, Q)$ denote the number of $(\vec{a}, q) \in \Z^3 \times \Z$ with $Q/2 \leq q \leq Q$ and $\vec{a}/q \in B_{\delta/Q}(\gamma)$.
\begin{theorem} \label{thm::strictcount}
Let $\gamma : [0, 1] \to \R^3$ be a non-degenerate curve.  For all $\delta \geq Q^{-1/3}$,
\[
\mathcal{N}(\delta, Q) \lsim \delta^2 Q^2.
\]
\end{theorem}
As noted before,  this is sharp because the total count gets dominated by the contribution of short line sections for $\tau>1/3$.  However,  as mentioned above,  it still makes sense to study the structure of the set of rationals up to height $Q$,  as for example encoded in the Hausdorff dimension.  With our method,  this seems viable up to the linearisation threshold at which $\sqrt{\eps} \eps^{n-1}$ is of order $Q^{-(n+1)}$ (the range conjectured by Beresnevich and Yang \cite{beresnevich-yang}),  up to which we can study the problem by solving the dual sparse counting problem.


\subsection{Structure of this paper}
The preparatory Sections \ref{sec::geomnumbers} and \ref{sec::detmethod} introduce the counting result that is the basis for the proof of Theorem \ref{thm::maindualitybox} and the determinant method lemma we frequently use. In Section \ref{sec::R2curves},  the results are combined with the Legendre dual to give sharp counting results for planar curves. Section \ref{sec::projduality} introduces the background in projective duality we need,  which is then used in Sections \ref{sec::dualboxes} and \ref{sec::dualcurves} to establish our duality principles. Section \ref{sec::geomlemmata} contains geometric lemmata needed in Section \ref{sec::R3curves},  and in Section \ref{sec::R3curves} we then combine the main duality result,  the determinant method and the geometric lemmata from the preceding section to show the sharp bound for space curves. 
\subsection{AI disclaimer and reflections}
This work was done without the use of AI except at the very end,  at which point Claude (Opus 5.5) was used for editing,  structuring and double-checking the proofs,  and for writing out some routine arguments I had left for later (like Lemma \ref{lem::Rnvolnghbhd}).  I have omitted to fix imprecisions in the proofs (like whether we are counting each rational point with multiplicity or not) in order to keep this paper as human as possible.

To be as transparent as possible,  I will give the full timeline:
I found the first proof for Theorem \ref{thm::R2curves} last year,  and by April of this year my understanding had matured to the streamlined proof in the introduction and an understanding that the ideas of the proof would apply to more general manifolds.  I talked about this at the British Mathematical Colloquium in April 2026.  I then proved Theorem \ref{thm::maindualitybox} and,  realising that there would be a way to prove Theorem \ref{thm::strictcount},  communicated this to the authors of \cite{chen-seeger-srivastava-technau} in May 2026.  

With the benefit of what was to come,  I should of course have wrapped up the project then and put the paper out in parallel to them,  but at the time,  it seemed sensible to try to avoid what would have amounted to a clean scoop.  To broaden the scope of this paper,  I pushed against the $\R^3$-case of the Beresnevich--Yang conjecture \cite[Conjecture 2.7]{beresnevich-yang} and the case of surfaces without positive Gaussian curvature in $\R^3$,  for both of which I saw a viable strategy and which I believe would have succeeded with a little more time.  While working through this,  I was overtaken by the speed of events and once AI models became as powerful as they are now,  I had to settle for pushing out what I had in hand when I could. I chose not to use AI models to help me bridge the gaps, but will leave this to a separate follow-up paper. 

I'll now reflect on the role I see for this paper and papers of its kind in the emerging mathematical universe dominated by AI.  While the results of this paper can be easily superseded with the powerful AI tools now available,  so can those of pretty much any current paper.  However,  I believe that there is value in this paper beyond this: It takes what were considered opaque papers and put them into a form that is both pretty and intuitive to understand.  Given the slope of improvement of the AI systems and the undoubtedly forthcoming ocean of proofs,  I believe that occasions to further the intuitive human understanding of theorems and proofs like in this paper will be one of the main motivators of research in pure mathematics.  In that sense,  even though the results are not what I envisioned completing two months ago,  I am proud of what I believe is the clear and intuitive point of view this paper transmits.  

\subsection{Acknowledgements}
I am grateful to Victor Beresnevich and the group in York for inviting me to the visit that sparked my exploration of this topic last year,  even though I had never worked in their area before.  I am also grateful to Jonathan Hickman for helpful conversations and his encouragement.  Lastly,  my thanks go to Rajula Srivastava and Christoph Thiele,  without whose broad interests I never would have learned about this problem.  

\section{The geometry of numbers} \label{sec::geomnumbers}
In this section,  we use several classical results from the geometry of numbers (see \cite[Chapter VIII]{cassels} for an introduction) and combine them to show a counting result that will be an important facet of our duality. 

We write $f \lsim g$ for quantities $f, g$ if there is a constant $C$ such that $f \leq C g$, here and everywhere else in the paper.  We write $f \sim g$ if $f \lsim g$ and $g \lsim f$.

For a convex distance function $F$, we write $F^*$ for the polar convex distance function defined by
\[
F^*(y):=\max_{x \in \R^n: F(x)=1} |\langle x, y \rangle|.
\]
When we speak of successive minima $\lambda_i$ of $F$ or the successive minima $\lambda_i^*$ of $F^*$ below, we always mean with respect to the lattice $\Z^n$.

The following theorem is a corollary of classical results in the geometry of numbers.  

\begin{theorem} \label{thm::polarcount}
Let $F$ be a symmetric convex distance function on $\R^n$ and let $\lambda_j$ be the successive minima of $F$ with respect to the lattice $\Z^n$.  Denote the successive minima of the lattice $\Z^n$ with respect to the polar convex distance function $F^*$ by $\lambda_i^*$. Then,
\[
\#\{ v \in \Z^n: F(v) \leq 1\} \sim \vol(F \leq 1 ) \prod_{j=1}^n \max\left\{\lambda_j,1 \right\} \sim \vol(F \leq 1 ) \prod_{j=1}^n \max\left\{\frac{1}{\lambda_j^*},1 \right\},
\]
where the implicit constants depend only on the dimension $n$. Moreover, if $\lambda_j^* \geq \eta^{-1}$ for all $j$,  then 
\[
\#\{ v \in \Z^n: F(v) \leq 1\}=(1+O(\eta)) \; \vol(F \leq 1 ).
\]
\end{theorem}
\begin{proof}
By Mahler's transference theorem \cite{mahler} (see also \cite[Chapter VIII]{cassels}),
\[
\lambda_j \sim (\lambda_{n-j+1}^*)^{-1},
\]
so the claim is equivalent to showing
\begin{equation} \label{eq::latticeest}
\#\{ v \in \Z^n: F(v) \leq 1\} \sim \vol(F \leq 1 ) \prod_{j=1}^n \max\left\{\lambda_j,1 \right\}.
\end{equation}
We first show (\ref{eq::latticeest}) in the case that $\lambda_j \leq 1$ for all $j$. In this case,  $\{F \leq 1\}$ contains a full fundamental domain of $\Z^n$ of volume one and is tiled by its translates. As each translate of the fundamental domain contains one point on average, the result follows.

In the case that $\lambda_j>1$ for some $j$,  we note that by Minkowski's second theorem \cite{minkowski} (see also \cite[Chapter VIII]{cassels}),
\[
\vol(F \leq 1 ) \prod_{j=1}^n \max\left\{\lambda_j,1 \right\} \sim \prod_{j=1}^n \max\left\{\lambda^{-1}_j,1 \right\}.
\]
By Lemma 2 in VIII.1.2 in \cite{cassels}, there is a basis $v_1, \dots, v_n \in \Z^n$ with $F(v_i) \sim \lambda_i$ such that any $v \in \Z^n \cap \{F \leq 1\}$ lies in $v_1 \Z+ \dots+v_r \Z$, where $r$ is chosen such that $\lambda_r \leq 1 < \lambda_{r+1}$.  We reduce to $\R^r \cong \Span(v_1, \dots, v_r)$ by defining the symmetric convex distance function $G$ on $\R^r$ as
\[
G(x):=F(x_1 v_1+ \dots+x_r v_r).
\] 
By the property of the $v_i$ find that showing (\ref{eq::latticeest}) reduces to  establishing
\[
\#\{ b \in \Z^r: G(b)\leq 1\} \sim \prod_{j=1}^r \lambda^{-1}_j= \prod_{j=1}^n \max\left\{\lambda^{-1}_j,1 \right\}.
\]
But by definition of $v_i$ and because $j \leq r$, the successive minima $\lambda_j^\prime$ of  $G$ with respect to $\Z^r$ satisfy $\lambda_j^\prime \sim F(v_j) \sim \lambda_j$.  By this and Minkowski's second theorem,
\[
\prod_{j=1}^r \lambda^{-1}_j \sim \prod_{j=1}^r (\lambda^\prime_j)^{-1} \sim \vol(G \leq 1),
\]
so we have reduced showing (\ref{eq::latticeest}) for $F$ to showing (\ref{eq::latticeest}) for $G$. But as $\lambda_j^\prime \lsim 1$ for all $j$, this follows from the case we already established.

Lastly,  the condition $\lambda_j^* \geq \eta^{-1}$ implies that $\lambda_j \leq C \eta$ for some $C$.  Thus,  any point $v \in \Z^n$ such that one of the cubes adjacent to $v$ partially intersects $\{F \leq 1\}$ must have $|F(v)-1| \leq \sqrt{n} C \eta$.  Consequently,
\[
\vol(F \leq 1-\sqrt{n} C\eta) \leq \#\{ v \in \Z^n: F(v) \leq 1\} \leq \vol(F \leq 1+\sqrt{n} C\eta),
\]
which implies the claim because 
\[
\vol(F \leq 1 \pm \sqrt{n}C\eta)=(1 \pm \sqrt{n} C\eta)^n \vol(F \leq 1)=(1+O(\eta)) \; \vol(F \leq 1).
\]
\end{proof}

Finally, to get the most mileage out of Theorem \ref{thm::polarcount},  the following Lemma is useful.

\begin{lemma} \label{lem::polardist}
Let $v_1, \dots, v_n$ be a basis of $\R^n$,  let $t_i \in \R_{>0}$ for $1 \leq i \leq n$ and let $F$ be the symmetric convex distance function defined by
\[
F(x)=\max_i|t_i \langle v_i, x \rangle|.
\]
For $y \in \R^n$,  let $\rho_j(y)$ be the real numbers such that 
\[
y=\sum_{j=1}^n \rho_j(y) v_j.
\]

Then, 
\[
F^*(y)=\sum_i \frac{1}{t_i}\left| \rho_i(y)\right|.
\]
\end{lemma}
\begin{proof}
Write $A$ for the matrix with $v_i$ in the rows, $D=\mathrm{diag}(t_1, \dots, t_n)$ and denote the maximum norm by $|\cdot|_\infty$.  By the definition of the dual norm,
\begin{align*}
F^*(y)&=\max_{x \neq 0} \frac{|\langle x, y \rangle|}{|Ax|_\infty}=\max_{x \neq 0} \frac{|\langle A^{-1} D^{-1} x, y \rangle|}{|x|_\infty}\\
&=\max_{x \neq 0}\frac{\left|\langle x, D^{-1}\left(A^{-1}\right)^T y \rangle \right|}{|x|_\infty}=\left| D^{-1} \left(A^{-1}\right)^T y \right|_1.
\end{align*}
As by definition $y=A^T \rho(y)$,  the claim follows.
\end{proof}

\section{The determinant method} \label{sec::detmethod}

We now show an observation related to the density of rationals in $\R^n$ with bounded height, which will be crucial to estimating the number of rational points near manifolds.  On $\R^2$,  an observation equivalent to this fact was first established by Huxley using the determinant method and proved to be the crucial ingredient in his treatment of curves.

We now restate and prove Lemma \ref{lem::ratrepellancy} from the introduction.
\begin{restated}{Lemma \ref{lem::ratrepellancy}}
Let $\vec{p}_1, \dots, \vec{p}_{n+1} \in \Z^n$ and $q_1, \dots, q_{n+1} \in \Z_{>0}$ such that $q_i \leq Q$ for some $Q$.  Define $v_i:=\vec{p}_i/q_i \in \R^n$.  Then
\[
\vol(\conv(v_1, \dots, v_{n+1})) \gsim Q^{-(n+1)}
\]
unless there is an affine hyperplane containing all $v_1, \dots, v_{n+1}$. 

Moreover,  if $H:=\{x: k_0+k_1 x_1+\dots+k_{n-1} x_{n-1}+m x_n=0\}$ is a rational hyperplane where $k_j$ and $m$ are integers with $\gcd(m, k_0, \dots, k_{n-1})=1$ and $m>0$, then for any collection of rational vectors $(\vec{a}_1/q_1,  b_1/q_1), \dots,  (\vec{a}_n/q_n,  b_n/q_n)$ contained in $H$,  where $\vec{a}_i \in \Z^{n-1}$ and again $q_i \leq Q$,
\[
\vol_{n-1}(\conv(\vec{a}_1/q_1, \dots, \vec{a}_n/q_n)) \gsim \frac{m}{Q^{n}}
\]
unless there is a subspace of dimension $n-2$ containing all of the rational points.  The implicit constants depend only on the dimension.
\end{restated}

\begin{remark} \label{rem::rathyperorthrepellancy}
An observation that we will frequently need is that if the rational hyperplane $H$ also repels rationals not lying on it in orthogonal direction.  Indeed,  if we assume that $H$ is orientated such that $m \sim \max(|k_1|,  \dots,  |k_{n-1}|,  m)$,  then for any $\vec{a}/q$ with $\vec{a} \in \Z^{n-1}$ and $q>0$ that does not lie on $H$, 
\[
|\langle \vec{a}/q,  \vec{n}_H \rangle| \gsim \frac{1}{qm}.
\]
This is quickly seen by comparing the expression above to the defining equation for $H$ and noting that any non-zero integer has size at least one.  
\end{remark}
\begin{proof}[Proof of Lemma \ref{lem::ratrepellancy}]
We observe that
\begin{align*}
\vol(\conv(v_1, \dots, v_{n+1}))&=\vol(\conv(0, v_1-v_{n+1},\dots, v_{n}-v_{n+1}))\\
&\sim \left|\det(v_1-v_{n+1}, \dots, v_{n}-v_{n+1}) \right|
\end{align*}
and
\begin{align*}
&\det(v_1-v_{n+1}, \dots, v_{n}-v_{n+1})=\det \begin{pmatrix} 
v_1-v_{n+1} & \dots & v_{n}-v_{n+1} & v_{n+1} \\ 0 & \dots & 0 & 1 
\end{pmatrix}\\
&=\det \begin{pmatrix} 
v_1& \dots & v_{n}& v_{n+1} \\ 1 & \dots & 1 & 1 \end{pmatrix} \\
&=q_1^{-1} \dots q_{n+1}^{-1} \det \begin{pmatrix} 
\vec{p}_1 & \dots & \vec{p}_{n}& \vec{p}_{n+1}\\ q_1 & \dots & q_{n} & q_{n+1} \end{pmatrix}.
\end{align*}
As the determinant of an integer matrix is an integer and in particular exceeds one in absolute value,  the claim follows unless 
\[
\det \begin{pmatrix} 
\vec{p}_1& \dots & \vec{p}_{n}& \vec{p}_{n+1}\\ q_1 & \dots & q_{n} & q_{n+1} \end{pmatrix}=0.
\]
By the same calculation,  this happens exactly when all of the $v_i$ are contained in an affine hyperplane.

The second part of the lemma follows in exactly the same way if we can show that $m$ divides $d:=\det \begin{pmatrix} 
\vec{a}_1& \dots & \vec{a}_{n}\\ q_1 & \dots & q_{n} \end{pmatrix}$.  This follows essentially because the condition of lying in $H$ makes the rows of the matrix linearly dependent over $\Z/m\Z$. However,  because $\Z/m\Z$ may contain zero divisors,  we need to take a little care to deduce $d=0 \mo m$ from this.

Write $m=\prod_i p_i^{r_i}$ and note that the proof is complete if we can show that $p_i^{r_i} \mid d$ for all primes $p_i$.  For any vector $(\vec{a}/q,  b/q) \in H$ we get that
\[
q k_0+k_1 a_1+\dots+ k_{n-1} a_{n-1}=-m b \in m\Z.
\]
Fix $p_i$,  let $s$ be an index such that $\gcd(k_s, p_i)=1$ and let $c$ be the inverse of $k_s$ in $\Z/p_i^{r_i}\Z$.  From the equation
\[
q_l k_0 +\sum_{j=1}^{n-1} (\vec{a}_l)_j k_j=0 \mo p_i^{r_i},
\]
which holds for all $l$ because the vectors $(\vec{a}_l/q_l,  b_l/q_l)$ are contained in $H$,  we get 
\[
(\vec{a}_l)_s=-c \left( q_l k_0 +\sum_{j \neq s} (\vec{a}_l)_j k_j \right) \mo p_i^{r_i}
\]
for all $l$.  Adding a multiple of one row to another does not change the determinant of a matrix with elements in $\Z/p_i^{r_i} \Z$,   so we get that in $\Z/p_i^{r_i} \Z$,  the determinant is equal to the determinant of the matrix with the $s$-th row replaced by zeroes.  Hence,  $d=0 \mo p_i^{r_i}$. 
\end{proof}
\section{Counting rational points near curves in $\R^2$} \label{sec::R2curves}
In this section,  we count rational points near sufficiently curved planar curves.  We fix an interval $I \subset \R$ and a twice continuously differentiable function $f \colon I \to \R$ with the property that $f^{\prime \prime}(x) \geq c>0$ for all $x \in I$. Define for any sub-interval $J \subset I$ the counting function 
\[
\mathcal{N}(\delta, Q, J):=\#\left\{(a, b, q) \in \Z^3,  \frac{1}{2} Q \leq q \leq Q: \; \frac{a}{q} \in J,  \; \left|f\left(\frac{a}{q}\right)-\frac{b}{q}\right| \leq \frac{\delta}{Q} \right\}
\]
and set
\[
\mathcal{N}(\delta, Q):=\mathcal{N}(\delta, Q, I).
\]
This counts the number of rationals points near the curve segment $\gamma(x)=(x, f(x)),  x \in I$ in $\R^2$.  

We restate Theorem \ref{thm::R2curves} from the introduction,  which establishes an asymptotic formula for rational points near curves under minimal smoothness assumptions,  and in particular gives upper and lower bounds in a larger than previously known range of $\delta$.
\begin{restated}{Theorem \ref{thm::R2curves}}
There is a constant $C$,  depending only on $f$,  such that for all $Q \geq 2$ and all $\delta \geq C Q^{-1}(\log Q)^2$,
\[
\mathcal{N}(\delta, Q)=\frac{7}{12}\,\delta Q^2 |I| \left(1+O\left(\frac{(\log Q)^{1/2}}{(\delta Q)^{1/4}}\right)\right).
\]
Moreover,  $\mathcal{N}(\delta, Q) \lsim \delta Q^2+Q(\log Q)^2$ for all $\delta>0$.  The implicit constants depend only on $f$.
\end{restated}
Theorem \ref{thm::R2curves} is a consequence of the following local version,  which shows that outside a small exceptional set,  the rational points near the curve are equidistributed down to scale $\sqrt{\delta \eta/Q}$.
\begin{theorem} \label{thm::R2local}
There is a constant $C$,  depending only on $f$,  such that the following holds for all $Q \geq 2$,  $\delta \geq C Q^{-1}(\log Q)^2$ and $\eta \in [(C (\log Q)^2/(\delta Q))^{1/4},  1]$.  There is a set $E \subset I$,  which is a union of $\lsim \delta^{-1}\eta^{-1}$ intervals,  with
\[
|E| \lsim \eta^{-1}(\delta Q)^{-1/2},
\]
such that for every interval $J \subset I \backslash E$,
\[
\mathcal{N}(\delta, Q, J)=\frac{7}{12}\,\delta Q^2 |J| \,(1+O(\eta))+O\left(\delta Q^2 \sqrt{\frac{\delta \eta}{Q}}\right).
\]
\end{theorem}
In \cite{huxley}, the upper bound in Theorem \ref{thm::R2curves} was established except for a $\delta^{-\eps}$ error in the main term.  This error term was subsequently removed in \cite{vaughan-velani}, where building on the result in \cite{huxley}, the upper bound in Theorem \ref{thm::R2curves} was established with an error term of roughly $Q (\log Q)^{200}$ instead of $Q (\log Q)^2$ under additional smoothness assumptions.  The lower bound in Theorem \ref{thm::R2curves} was first established in \cite{beresnevich-dickinson-velani}, also under additional smoothness assumptions.  An asymptotic formula was previously established by Huang in \cite{huangasymp} in the weaker range of $\delta \geq Q^{-1+\eps}$ through Fourier analysis.  Huang also works with a modified counting function and emphasises that his method does not allow to treat the natural counting function $\mathcal{N}(\delta, Q)$,  while our methods work for both their and our counting function.  Moreover,  all of these results rely on Huxley's bound in \cite{huxley},  which we bypass by using Lemma \ref{lem::ratrepellancy} directly. 

As we will see,  Theorem \ref{thm::R2curves} follows from a corresponding upper bound for rationals near curves in the sparse range of $\delta < Q^{-1}$. 
\begin{theorem} \label{thm::sparsecurve}
There is a constant $C$ depending only on the curve such that for any $Q$ and any $\delta \leq \frac{1}{ C Q (\log Q)^2}$, 
\[
\mathcal{N}(\delta, Q) \leq C Q.
\]
\end{theorem}
\subsection{Auxiliary results}
We first show that the rational points near the curve can be covered by few lines.
\begin{lemma} \label{lem::linecover}
Let $f \colon I \to \R$ be twice continuously differentiable on a bounded interval $I$ and let $(x, f(x))$ be the corresponding curve in $\R^2$.  Let $Q \in \N$.  There is $N \lsim Q$ and affine lines $l_1, \dots, l_N$ containing all rational points of the form $(a/q, b/q)$ with $1 \leq q \leq Q$ and 
\[
\left|f\left(\frac{a}{q}\right)-\frac{b}{q}\right| \leq \frac{1}{Q^2}.
\]
\end{lemma}
\begin{proof}
Let three points $(a_i/q_i,  b_i/q_i)$,  $i=1, 2, 3$,  of the form in the lemma be given.  By Lemma \ref{lem::ratrepellancy},  the three points are either co-linear or enclose an area of size $\gsim Q^{-3}$.  If the three points satisfy 
\[
\max_{i, j} |a_i/q_i-a_j/q_j| \leq \frac{c}{Q},
\]
for some small $c>0$,  then as 
\[
f(x)=f(x_0)+(x-x_0) f^\prime(x_0)+O\left(\frac{1}{Q^2}\right)
\]
around $x_0:=a_1/q_1$,  all $(a_i/q_i,  b_i/q_i)$ are contained in a rectangular strip of length $\frac{c}{Q}$ and width $O\left( \frac{1}{Q^2} \right)$.  Thus,
\[
\vol(\conv((a_i/q_i,  b_i/q_i),  i=1, 2, 3)) \lsim \frac{c}{Q^3},
\]
which is less than the threshold prescribed by Lemma \ref{lem::ratrepellancy} provided that $c$ is sufficiently small.  Thus,  we must have 
\[
\max_{i, j} |a_i/q_i-a_j/q_j| \geq \frac{c}{Q}
\]
unless the three points are co-linear.  

To cover the set $V$ of rational points $\{v:=(a/q, b/q)\}$ satisfying the assumption in the lemma,  define $\pi \colon V \to \R$ by $\pi(v)=a/q$ and let $v_1$ and $v_1^\prime$ be two points in $V$ minimizing $\pi$. Let $l_1$ be the line passing through them, let $V_1$ be the set of points lying on $l_1$,  let $v_2, v_2^\prime$ be two points in $V-V_1$ minimizing $\pi$ and define $l_2$ to be the line passing through them.  The points $v_1, v_1^\prime$ and $v_2$ are not co-linear,  so by the above,  necessarily 
\[
\pi(v_2)-\pi(v_1) \geq \frac{c}{Q}.
\]
Continuing this procedure to define $v_i, v_i^\prime$ and $l_i$,  we see that we cover all points in $V$ and that as $\pi(v_{i+1})-\pi(v_i) \geq \frac{c}{Q}$ for all $i$,  we do so with at most $|I| c^{-1} Q+1$ many lines.
\end{proof}

Finally,  we need the following lemma on rational lines near the curve.
\begin{lemma} \label{lem::linecount}
Let $f \colon I \to \R$ be twice continuously differentiable with $c \leq f^{\prime \prime} \leq c^{-1}$ on $I$,  let $\gamma(x)=(x, f(x))$ and let $(y, g(y))$,  $y \in f^\prime(I)$,  be the dual curve,  where $g(f^\prime(x))=x f^\prime(x)-f(x)$.  Let $\eps>0$ be small and let $l=\{(x,y): nx+my+k=0\}$ with $(n, m, k) \in \Z^3$ primitive,  $m \geq 1$ and $|n/m| \lsim 1$.  Assume that $d:=\mathrm{dist}(l, \gamma) \leq \eps$ and that this distance is attained at some $x_0$ not too close to the boundary of $I$,  and let $\theta:=\left|\frac{-n}{m}-f^\prime(x_0)\right|$ be the angle between the line and the tangent of the graph at $x_0$.  Then 
\begin{equation} \label{eq::linedens}
\#\left\{\left(\frac{a}{q}, \frac{b}{q} \right) \in l \cap B_{\eps}(\gamma): 1 \leq q \leq Q \right\} \lsim 1+\frac{Q^2}{m} \sqrt{\eps} \min\left\{\theta^{-1} \sqrt{\eps}, 1 \right\},
\end{equation}
where $B_\eps(\gamma)$ denotes the $\eps$-neighbourhood of the curve and we count distinct rational points.  Moreover,  the point $\left(\frac{-n}{m}, \frac{k}{m} \right)$ is at distance $\lsim \theta^2+d$ from the dual curve $(y, g(y))$.  The implicit constants depend only on $c$.
\end{lemma}

\begin{proof}[Proof of Lemma \ref{lem::linecount}]
By the second part of Lemma \ref{lem::ratrepellancy},  
\begin{equation} \label{eq::linedist}
\left|\frac{a_1}{q_1}-\frac{a_2}{q_2} \right|\geq \frac{m}{Q^2}
\end{equation}
for distinct rational points $(a_i/q_i, b_i/q_i)$ lying on the line.  We deduce that 
\[
\#\left\{\left(\frac{a}{q}, \frac{b}{q} \right) \in l \cap B_{\eps}(\gamma): 1 \leq q \leq Q \right\} \lsim 1+\frac{Q^2}{m} \mathrm{length}(l \cap B_{\eps}(\gamma)).
\]
Using Taylor approximation around $x_0$ and recalling that the curvature condition implies that $f^\prime$ is monotonously increasing with a given rate, we see that
\[
\mathrm{length}(l \cap B_{\eps}(\gamma)) \lsim \mathrm{length}(l \cap B_{\eps}\left(I_0\right)),
\]
where $I_0:=\left\{(x_0+y,  f(x_0)+f^\prime(x_0)(y-x_0)): |y-x_0| \lsim \sqrt{\eps} \right\}$. As the angle of $l$ is $\frac{-n}{m}$ and $B_{\eps}\left(I_0\right)$ is a rectangle of width $\eps$ and length $\sqrt{\eps}$, the inequality 
\[
\mathrm{length}(l \cap B_{\eps}\left(I_0\right)) \lsim \min\left\{\theta^{-1} \eps , \sqrt{\eps} \ \right\}
\]
follows. We thus have established (\ref{eq::linedens}).

To estimate the distance of $\left(\frac{-n}{m}, \frac{k}{m} \right)$ from the dual curve,  
we let $x_1$ be the point such that $f^\prime(x_1)=\frac{-n}{m}$ and note that by the curvature condition, $|x_1-x_0| \sim \theta$, which implies 
\[
f(x_0)=f(x_1)+(x_0-x_1) f^\prime(x_1)+O\left(\theta^2 \right)=f(x_1)+(x_0-x_1) \frac{-n}{m}+O\left(\theta^2 \right).
\]
Recalling that also
\[
\left|f(x_0)+\frac{n}{m}x_0+\frac{k}{m}\right| \lsim d
\]
by the definition of $d$,  we find
\begin{align*}
\left|\frac{k}{m}-g\left(\frac{-n}{m}\right)\right|=\left|\frac{k}{m}+f(x_1)-\frac{-n}{m} x_1 \right| \lsim \theta^2+ \left|\frac{k}{m}+f(x_0)-\frac{-n}{m} x_0 \right|\lsim \theta^2 + d,
\end{align*}
concluding the proof.
 \end{proof}

\subsection{Proof of Theorem \ref{thm::sparsecurve}}
\begin{proof}[Proof of Theorem \ref{thm::sparsecurve}]
Let $C$ be a large constant and assume that 
\[
\mathcal{N}(\delta, Q) \geq C Q.
\]
We will drive this assumption to contradiction by finding a density increment: That is, we show that, provided that $C$ is large enough for $C^{1/4}$ to beat the arising implicit constants, there is a $Q^\prime \leq Q$ and a $\delta^\prime \leq \frac{1}{C^2 Q^\prime (\log Q^\prime)^2}$ such that 
\[
\mathcal{N}(\delta^\prime, Q^\prime) \geq C^{5/4} Q^\prime .
\]
Iterating this result gives the desired contradiction. 

To get the proof started,  we use Lemma \ref{lem::linecover}  to find lines $l_1, \dots, l_N$ with $N \lsim Q$ covering all of the points counted in $\mathcal{N}(\delta, Q)$.  Let each individual line be defined by the equation $n_j x+m_j y+k_j=0$, where we take $m_j\geq 0$.  Let $x_j$ be the point minimizing the distance between $l_j$ and $\{(x, f(x))\}$ and let 
\[
\theta_j:=\left|\frac{-n_j}{m_j}-f^\prime(x_j)\right|
\]
be the angle between the line and the tangent to the curve.  Any line with $m_j=0$ has $\theta_j \gsim 1$ and can thus contribute at most $Q^2 \theta_j^{-1} \delta/Q \lsim 1$ many points to the total count.  Hence, we can assume $m_j>0$.  

To estimate the total contribution of the other lines,  we let 
\[
t_j:=\max\left\{1,  \theta_j \sqrt{\frac{Q}{\delta}} \right\}. 
\]
By Lemma \ref{lem::linecount},
\[
C Q  \leq \mathcal{N}(\delta, Q) \lsim  Q^2 \sqrt{\frac{\delta}{Q}} \sum_j \frac{1}{t_j m_j} \leq \frac{Q}{\sqrt{C} \log Q} \sum_j \frac{1}{t_j m_j}.
\]
Because there are only $\lsim Q$ many indices,  the contribution of all $j$ such that $t_j m_j \geq \frac{Q}{C^{5/4} \log Q}$ is at most $CQ/2$, provided that $C^{1/4}$ beats the implicit constant.  Thus, it must already hold that 
\[
\sum_{j, t_j m_j \leq \frac{Q}{C^{5/4} \log Q}} \frac{1}{t_j m_j} \gsim C^{3/2} \log Q.
\]

Put the indices $j$ with $t_j m_j \leq \frac{Q}{C^{5/4} \log Q}$ into classes $A_l$ according to the size of $m_j$.  That is, let $A_l$ consist of indices with $M_l \leq m_j \leq 2 M_l$ for $j \in A_l$, where $M_l:=2^{l-1}$,  and let $C_l$ be the number of indices in class $A_l$.  There are only $\log Q$ many classes and $t_j \geq 1$ for all $j$, so again assuming that $C^{1/4}$ beats the implicit constants,  there must be an $l$ such that
\[
C_l \geq C^{5/4} M_l.
\]

By the second part of Lemma \ref{lem::linecount},  the distance of the rational $\left(\frac{-n_j}{m_j}, \frac{k_j}{m_j}\right)$ to the dual curve is proportional to $t_j^2 \delta/Q$. As $t_j m_j$ is not too big for the indices we are considering,  this distance is bounded by
\[
\frac{t_j^2 \delta}{Q} \leq \frac{t_j ^2}{C Q^2 (\log Q)^2} \leq \frac{(t_j m_j)^2}{C Q^2 (\log Q)^2  M_l^2}  \leq \frac{1}{C^{7/2} (\log Q)^4 M_l^2}.
\]
Putting $Q^\prime=M_l$ and $\delta^\prime:=\frac{1}{C^{7/2} (\log Q)^4 M_l}$, 
we see that we have indeed arrived at an instance 
of the counting problem for the dual curve $g$ in which 
\[
\mathcal{N}\left(\delta^\prime, Q^\prime, g \right) \geq C^{5/4} Q^\prime.
\]
As the dual of $g$ is our original curve $f$, we can iterate this as long as we like (without worrying about implicit constants blowing up), eventually arriving at a contradiction. 
\end{proof}

\subsection{Proof of Theorems \ref{thm::R2local} and \ref{thm::R2curves}}
We now see how the result in the sparse setting translates into the general one.  Both theorems follow from the proposition below.  Throughout,  $C$ is a large constant depending only on $f$,  $\delta \geq C Q^{-1}(\log Q)^2$ and $\eta \in [(C(\log Q)^2/(\delta Q))^{1/4}, 1]$.  We let $x_j$ be the $j$-th element of $I \cap \sqrt{\frac{\delta \eta}{Q}}\, \Z$ and $J_j$ the interval of length $\sqrt{\frac{\delta\eta}{Q}}$ centred at $x_j$,  and we write $M_l:=2^l$.
\begin{prop} \label{prop::R2structure}
\begin{enumerate}
\item[(i)] For each $l$ with $M_l \leq \delta^{-1}\eta^{-1}$,  there are $O(M_l)$ rational lines
\[
\ell^{(l)}_i=\{(x, y): n x+m y+k=0\},  \qquad (n, m, k) \in \Z^3 \textrm{ primitive},  \quad M_l<m \leq 2M_l,
\]
such that,  writing $x^{(l)}_i \in I$ for a point at which $\ell^{(l)}_i$ comes closest to the curve $(x, f(x))$,
\[
\mathcal{N}(\delta, Q, J_j)=\frac{7}{12}\,\delta Q^2 |J_j| \,(1+O(\eta))
\]
for every $j$ with $|x_j-x^{(l)}_i| \geq (\eta M_l Q)^{-1/2}$ for all $l$ and $i$.
\item[(ii)] Let $\eta=1$.  For each $l$ with $M_l \leq \delta^{-1}$ and each $i$,  let $I^{(l)}_i$ be the set of $x \in I$ at which $\ell^{(l)}_i$ meets the $\frac{\delta}{Q}$-neighbourhood of the curve.  Then $|I^{(l)}_i| \lsim \sqrt{\delta/Q}$,
\[
\sum_i \mathcal{N}(\delta, Q, I^{(l)}_i) \lsim \delta^{1/2} Q^{3/2},
\]
and $\mathcal{N}(\delta, Q, J_j) \lsim \delta Q^2|J_j|$ for every $J_j$ disjoint from all $I^{(l)}_i$.
\end{enumerate}
\end{prop}

\begin{proof}[Proof of Proposition \ref{prop::R2structure}]
To find the $x^{(l)}_i$,  we use the points $x_j$ defined above.  For each index $j$, define a symmetric convex distance function $F_j$ on $\R^3$ by
\[
F_j((a, b, q))=\max\left(\frac{1}{Q}|q|,  \frac{1}{\sqrt{\eta \delta Q}} |a - x_j q|, \frac{1}{\delta} \left|b-a f^\prime(x_j)- q [f(x_j)-x_j f^\prime(x_j)] \right| \right).
\]
This function is intimately related to counting the rationals near the curve on the interval $J_j$ because by Taylor approximation,
\[
f(x)=f(x_j)+(x-x_j) f^\prime(x_j)+O\left(\frac{\eta \delta}{Q}\right) 
\]
for $x \in J_j$.  Up to an $O(\eta)$ error  counting the rational points near the curve will correspond to counting integer points with $F_j(v) \leq 1$.  To be precise, given a vector $(a,b, q)$ of integers with $a/q \in J_j$ and $\frac{1}{2}Q \leq q \leq Q$,  the point $(a/q, b/q)$ will surely lie in the $\frac{\delta}{Q}$-neighborhood of the curve if $F_j(a, b, q) \leq 1-O(\eta)$ and will surely not lie in the neighborhood if $F_j(a, b, q) \geq 1+O(\eta)$ (we will use this later on).

We call the interval $J_j$ \emph{good} if $\lambda_3(F_j)\leq \eta$, where $\lambda_3$ is the third successive minimum of $F_j$ with respect to the lattice $\Z^3$.  We will control for how many indices this fails, and the number of rationals with $a/q$ lying in one of the problem intervals, by applying Theorem \ref{thm::sparsecurve} to the dual curve $g$.  

To control the occurrences at which $\lambda_3(F_j)> \eta$, we note that by Mahler's transference theorem \cite{mahler},  
\[
\lambda_3(F_j) \sim \lambda_1(F_j^*),
\] 
where $F_j^*$ is the dual distance function to $F_j$.  We have that
\[
\begin{pmatrix}
n \\ m \\ k
\end{pmatrix}=(k+n x_j+m f(x_j)) \begin{pmatrix}
0 \\ 0 \\ 1
\end{pmatrix}+(n+m f^\prime(x_j))\begin{pmatrix}
1 \\ 0 \\ -x_j
\end{pmatrix}+m \begin{pmatrix}
-f^\prime(x_j) \\ 1 \\ -f(x_j)+x_j f^\prime(x_j)
\end{pmatrix},
\]
so by Lemma \ref{lem::polardist},
\begin{equation} \label{eq::polarFjdef}
F_j^*((n, m, k))=\delta |m|+\sqrt{\eta \delta Q} |n+m f^\prime(x_j)|+Q|k+n x_j+m f(x_j)|.
\end{equation}
We denote by $(n_j, m_j, k_j)$ the vector in $\Z^3 \backslash \{0\}$ minimizing $F_j^*$, where we use the symmetry to take $m_j \geq 0$, and note that because of our assumption on $\eta$,  $F_j^*((n_j, m_j, k_j)) \leq \eta^{-1}$ is possible only if $m_j \neq 0$ and $m_j \leq \delta^{-1} \eta^{-1}$.  To bound the number of problematic intervals, it thus suffices to consider $j$ with $1\leq m_j \leq \delta^{-1} \eta^{-1}$.  Define the set 
\[
\mathcal{A}:=\{(n_j, m_j, k_j),  1 \leq m_j \leq \delta^{-1} \eta^{-1} \}. 
\]
We put the problematic indices into classes $A_l$,  where $l$ is such that $M_l:=2^l \leq \delta^{-1} \eta^{-1}$,  by defining
\[
A_l:=\{(n_j, m_j, k_j) \in \A: M_l<m_j \leq 2M_l\}.
\]
Fix $l$ and define for each $(n_j, m_j, k_j) \in A_l$ the corresponding line $\ell^{(l)}_i:=\{(x, y):n_j x+m_j y+k_j=0\}$ and the point $x^{(l)}_i$ at which it comes closest to the curve $(x, f(x))$. 

We first show that for each $x$ that is at a distance of at least $\Omega(\frac{1}{\sqrt{\eta M_l Q}})$ from each $x^{(l)}_i$, we have that $\lambda_3(F_j) \leq \eta$.  To see this,  we let $x_*$ be the point such that $f^\prime(x_*)=\frac{-n}{m}$ and observe that
\[
n+m f^\prime(x) \sim m |x-x_*| \sim M_l |x-x_*|.
\]
Thus, by monotonicity of $f^\prime$,
\[
\left| k+n x+mf(x)-(k+nx_*+mf(x_*))\right|=\left|\int_{x_*}^x n+mf^\prime(x) dx \right| \gsim M_l (x -x_*)^2,
\]
which implies that 
\[
\eta^{-1} \leq Q|k+n x+m f(x)| 
\]
can only happen on an interval of length $O\left(\frac{1}{\sqrt{\eta M_l Q}}\right)$. As 
\[
F_j^*((n, m, k))\geq Q|k+n x_j+m f(x_j)|
\]
by definition of $F_j^*$,  we see that the points with $F_j^*(n, m, k) \geq \eta^{-1}$ are indeed contained in an interval of size bounded by $O\left(\frac{1}{\sqrt{\eta M_l Q}}\right)$.  We have thus shown that indeed $\lambda_3(F_j) \leq \eta$ provided that $|x_j-x^{(l)}_i| \geq (\eta M_l Q)^{-1/2}$ for all $l$ and $i$. 

To make use of the condition $\lambda_3(F_j) \leq \eta$,  we let $j$ be an index such that $\lambda_3(F_j)\leq \eta$, or equivalently, such that there are three linearly independent vectors $w_1, w_2, w_3 \in \Z^3$ with $F_j(w_i) \leq \eta$ for $i=1, 2, 3$.  Let $A$ be the linear transformation 
\[
(a, b, q) \mapsto  \left(\frac{a - x_j q}{\sqrt{\eta \delta Q}}, \frac{b-a f^\prime(x_j)- q [f(x_j)-x_j f^\prime(x_j)]}{\delta}, \frac{q}{Q} \right),
\]
or in matrix form,
\[
A=\begin{pmatrix}
\frac{1}{\sqrt{\eta \delta Q}} & 0 & \frac{-x_j}{\sqrt{\eta \delta Q}} \\ \frac{-f^\prime(x_j)}{\delta} & \frac{1}{\delta}& \frac{-f(x_j)+x_j f^\prime(x_j)}{\delta} \\ 0 & 0 & \frac{1}{Q}
\end{pmatrix}.
\]
Then,  $F_j(A^{-1}x)=|x|_\infty$,  where $|\cdot|_\infty$ is the usual maximum norm.
Define $\Lambda=A \Z^3$ and note that by the property that $\lambda_3(F_j) \leq \eta$,  the lattice $\Lambda$ contains three independent vectors $v_i=Aw_i$ with $|v_i|_\infty \leq \eta$ and thus has a fundamental domain of diameter $O(\eta)$.  Under the transformation $A$, the domain 
\[
\left\{(a, b, q) \in \R^3: \frac{1}{2}Q \leq q \leq Q,  |a-q x_j| \leq \frac{q\sqrt{\delta \eta}}{2 \sqrt{Q}},  |b-a f^\prime(x_j)- q [f(x_j)-x_j f^\prime(x_j)]| \leq \frac{\delta q }{Q} \right\}
\]
gets mapped to the domain 
\[
\Delta:=\left\{(x, y, z) \in \R^3: \frac{1}{2} \leq x \leq 1,  \frac{-x}{2} \leq y \leq \frac{x}{2},  -x \leq z \leq x \right\}.
\]
As the fundamental domain of $\Lambda$ has a size of at most $O(\eta)$ and $\Delta$ is of much bigger size, 
\[
\#\{v \in \Lambda \cap \Delta\}=\frac{\vol(\Delta)}{\mathrm{covol}(\Lambda)}(1+O(\eta))=\frac{7}{12}\delta Q^2 |J_j| (1+O(\eta)),
\]
where we used
\[
\mathrm{covol}(\Lambda)=\det(A)=\frac{1}{\delta Q^2 |J_j|}
\]
and
\[
\vol(\Delta)=\int_{1/2}^1 2x^2 dx=\frac{7}{12}.
\]
Finally,  to relate this to our counting problem, we recall that by the Taylor approximation argument,  any point in $\Lambda$ inside the domain that is at least a distance of $O(\eta)$ from the boundary corresponds to a rational in the $\frac{\delta}{Q}$-neighbourhood of the curve on the interval $J_j$,  and no lattice point that is further away than $O(\eta)$ from $\Delta$ does.  The domain we would really need to estimate in the lattice count is some curly domain $\widetilde{\Delta}$, but by this argument, we know that $\widetilde{\Delta}$ agrees with $\Delta$ everywhere except in the $O(\eta)$-neighbourhood of the boundary of $\Delta$, which we denote by $H$.  We then find
\begin{align*}
\mathcal{N}(\delta, Q, J_j)&=\#\{v \in \Lambda \cap \widetilde{\Delta}\}=\#\{v \in \Lambda \cap \Delta\}+O\left(\#\{v \in \Lambda \cap H \} \right)\\
&=\frac{7}{12}\delta Q^2 |J_j| (1+O(\eta)),
\end{align*}
which concludes the analysis of the good intervals and proves part (i). 

We now show that for each $l$ there are only $O(M_l)$ many $x^{(l)}_i$, which is the crux of the proof. For this,  we apply Theorem \ref{thm::sparsecurve} to the dual curve.

We fix an $l$ and a $i$ and assume that $x^{(l)}_i$ lies in the interval around some point $x_j$.  We define a symmetric convex distance function $G_j$, coming from the dual curve $(y,  g(y))$, with parameters $\delta^\prime$ and $Q^\prime$ on an interval of length $\sqrt{\frac{ \eta \delta^\prime}{Q^\prime}}$ around the point $y_j:=f^\prime(x_j)$ by
\begin{align*}
G_j((n, k, m))&=\max\left(\frac{1}{Q^\prime} |m|,  \frac{1}{\sqrt{\eta \delta^\prime Q^\prime}}|n-m y_j|,  \frac{1}{\delta^\prime} |k-n g^\prime(y_j)-m [g(y_j)-y_j g^\prime(y_j)]|\right)\\
&=\max\left(\frac{1}{Q^\prime} |m|,  \frac{1}{\sqrt{\eta \delta^\prime Q^\prime}}|n-m f^\prime(x_j)|,  \frac{1}{\delta^\prime} |k-n x_j+m f(x_j)|\right).
\end{align*}
If we assume $\eta=1$ for a moment,  then we see that with the choice of $\delta^\prime=\frac{1}{Q}$ and $Q^\prime=\delta^{-1}$, 
\[
F_j^*((n, m, k)) \sim G_j((-n, k, m)),
\]
so the vector $(n_j, m_j, k_j)$ at which $F_j^*$ is small corresponds to a rational point $(-n_j/m_j,  k_j/m_j)$ lying extraordinarily close (as $\delta^\prime<\frac{1}{Q^\prime}$) to the dual curve.  

It should not be surprising that Theorem \ref{thm::sparsecurve} comes into play, now.  
We set $Q^\prime=M_l$ and $\delta^\prime:=\frac{1}{C M_l (\log Q)^2}$ (where $C$ is the absolute constant from Theorem \ref{thm::sparsecurve}).  We get that there are at most $O(M_l)$ many combinations of integers $(n, m, k)$ with $M_l<m \leq 2M_l$ such that $G_j(-n, k, m) \lsim 1$.  It only remains to check that with these choices,  $G_j(-n, k, m) \lsim 1$ is implied by $F_j^*(n, m, k) \lsim \eta^{-1}$.  This follows because
\[
\frac{1}{\sqrt{\eta \delta^\prime Q^\prime}}=\frac{\sqrt{C} \log Q}{\sqrt{\eta}} \leq \eta \sqrt{\eta \delta Q}
\]
and 
\[
\frac{1}{\delta^\prime}=C M_l (\log Q)^2 \leq C \delta^{-1} \eta^{-1} (\log Q)^2  \leq \eta Q
\]
by our assumptions on $\eta$. 

To complete the proof,  we still have to estimate the contribution of rationals near each individual $x^{(l)}_i$,  which is fortunately not too hard.  Recall that $I^{(l)}_i$ is the set in which the line $\ell^{(l)}_i$ meets the $\frac{\delta}{Q}$-neighbourhood of the curve $(x, f(x))$.  By the curvature condition,  $|I^{(l)}_i| \lsim \sqrt{\frac{\delta}{Q}}$.  By Lemma \ref{lem::linecount},  the line $\ell^{(l)}_i$ contains $\lsim 1+\frac{Q^2}{M_l} \sqrt{\frac{\delta}{Q}}$ rational points of height at most $Q$ near the curve,  and as there are $O(M_l)$ lines for each $l$ and $M_l \leq \delta^{-1} \leq \delta^{1/2}Q^{3/2}$,  their total number is $\lsim \delta^{1/2}Q^{3/2}$.  This gives
\[
\sum_i \mathcal{N}(\delta, Q, I^{(l)}_i) \lsim \delta^{1/2} Q^{3/2}.
\]
The last assertion of part (ii) is also not hard to see: If we fix an interval $J_j$ in the partition that does not intersect any $I^{(l)}_i$,  then if $F_j^*((n, m, k)) \geq c_0$ (for some absolute $c_0=c_0(n)$) for all $(n, m, k) \in \Z^3 \backslash \{0\}$, 
\[
\mathcal{N}(\delta, Q, J_j) \lsim \delta Q^2 |J_j|
\]
by Theorem \ref{thm::polarcount}. If there is a $(n, m, k)$ with $F_j^*((n, m, k)) \leq c_0$, then each rational point $(a/q, b/q)$ with $F_j(a, b, q) \leq 1$ lies on the line $l=\{(x, y): nx+my+k=0\}$ corresponding to $(n, m, k)$ (provided that $c_0$ is small enough to beat the constant in Mahler's Theorem). However,  as $F_j^*((n, m, k)) \leq c_0$ implies $m \leq c_0^{-1} \delta^{-1}$, the line $l$ already corresponds to some $x^{(l)}_i$. Because the $J_j$ is disjoint from the interval $I^{(l)}_i$ in which the line intersects the $\frac{\delta}{Q}$-neighbourhood of the curve,  in this case, there are no rational points near the curve  at all. 
\end{proof}
\begin{proof}[Proof of Theorem \ref{thm::R2local}]
Let $x^{(l)}_i$ be the points from Proposition \ref{prop::R2structure}(i) and let $E$ be the set of $x \in I$ with $|x-x^{(l)}_i|<2(\eta M_l Q)^{-1/2}$ for some $l$ and $i$.  Then $E$ is a union of $\lsim \sum_{l: M_l \leq \delta^{-1}\eta^{-1}} M_l \lsim \delta^{-1}\eta^{-1}$ intervals and
\[
|E| \lsim \sum_{l: M_l \leq \delta^{-1}\eta^{-1}} M_l \,(\eta M_l Q)^{-1/2} \lsim \eta^{-1}(\delta Q)^{-1/2}.
\]
Let $J \subset I \backslash E$ be an interval.  As $M_l \leq \delta^{-1}\eta^{-1}$ implies $\delta \eta^2 M_l \leq \eta \leq 1$,  we have $\sqrt{\delta\eta/Q} \leq (\eta M_l Q)^{-1/2}$,  so the centre $x_j$ of every interval $J_j$ meeting $J$ satisfies $|x_j-x^{(l)}_i| \geq (\eta M_l Q)^{-1/2}$ for all $l$ and $i$.  By Proposition \ref{prop::R2structure}(i),  the asymptotic formula therefore holds for all $J_j$ meeting $J$.  As $J$ is the union of the $J_j$ contained in it and of parts of at most two further intervals $J_j$,  each containing $\lsim \delta Q^2 \sqrt{\delta\eta/Q}$ points,  the claim follows.
\end{proof}

\begin{proof}[Proof of Theorem \ref{thm::R2curves}]
We first show the upper bound.  As the counting function is monotone in $\delta$,  it suffices to consider $\delta \geq C Q^{-1}(\log Q)^2$.  Apply Proposition \ref{prop::R2structure}(ii).  The intervals $J_j$ disjoint from all $I^{(l)}_i$ contribute $\lsim \delta Q^2|I|$.  For fixed $l$,  Proposition \ref{prop::R2structure}(ii) gives
\[
\sum_i \mathcal{N}(\delta, Q, I^{(l)}_i) \lsim \delta^{1/2} Q^{3/2},
\]
and summing over the $\lsim \log Q$ values of $l$ gives $\lsim \delta^{1/2}Q^{3/2}\log Q \leq \delta Q^2$,  as $\delta Q \geq (\log Q)^2$.  This shows $\mathcal{N}(\delta, Q) \lsim \delta Q^2$ for $\delta \geq CQ^{-1}(\log Q)^2$,  and hence $\mathcal{N}(\delta, Q) \lsim Q(\log Q)^2$ for smaller $\delta$.

For the asymptotic formula,  let $\eta:=(C(\log Q)^2/(\delta Q))^{1/4} \leq 1$ and let $E$ be as in Theorem \ref{thm::R2local}.  The set $I \backslash E$ is a union of $K \lsim \delta^{-1}\eta^{-1}$ intervals,  so Theorem \ref{thm::R2local} gives
\[
\mathcal{N}(\delta, Q, I \backslash E)=\frac{7}{12}\,\delta Q^2 |I \backslash E| \,(1+O(\eta))+O\left(K \delta Q^2 \sqrt{\frac{\delta\eta}{Q}}\right),
\]
where $|I \backslash E|=|I|+O(\eta^{-1}(\delta Q)^{-1/2})$ and $K\sqrt{\delta\eta/Q} \lsim \eta^{-1/2}(\delta Q)^{-1/2}$.  To bound the points over $E$,  we cover $E$ by the intervals $J_j$ of Proposition \ref{prop::R2structure}(ii) (with $\eta=1$) that meet $E$.  As in the proof of the upper bound,  those disjoint from all $I^{(l)}_i$ contribute $\lsim \delta Q^2(|E|+K\sqrt{\delta/Q}) \lsim \delta Q^2 \eta^{-1}(\delta Q)^{-1/2}$,  and the intervals $I^{(l)}_i$ contribute $\lsim \delta^{1/2}Q^{3/2}\log Q$ in total.  Altogether,
\[
\mathcal{N}(\delta, Q)=\frac{7}{12}\,\delta Q^2 |I| \left(1+O\left(\eta+\eta^{-1}(\delta Q)^{-1/2}+(\delta Q)^{-1/2}\log Q\right)\right).
\]
As $\eta^4 \delta Q=C(\log Q)^2 \geq 1$ and $\delta Q \geq (\log Q)^2$,  both $\eta^{-1}(\delta Q)^{-1/2}$ and $(\delta Q)^{-1/2}\log Q$ are $\lsim \eta \lsim (\log Q)^{1/2}(\delta Q)^{-1/4}$,  which proves the claim.
\end{proof}


\section{Projective Duality} \label{sec::projduality}

\subsection{Basics} \label{subsec::basics}
For hyperplanes in $\R^n$, a fixed direction (say the one in the $n$-th coordinate direction) and $c>0$,   projective duality associates to each affine hyperplane $H \subset \R^n$ with $\langle \vec{n_H},  e_n \rangle \geq c$ a point in the projective dual.  Precisely,  the point $x_H \in \R^n$ associated to the hyperplane $H$ is given by
\[
x_H=(\underline{n_H},  -\langle y_H,  (\underline{n_H}, 1) \rangle),
\]
where $\underline{n_H}$ is the vector such that $(\underline{n_H}, 1)$ is normal to $H$ and $y_H$ is any point in $H$.  With these definitions,  the projective dual is isomorphic to $\R^n$. 
 
Let $\gamma$ be a $n$ times continuously differentiable non-degenerate curve in $\R^n$;  that is, it has the property that $\dot{\gamma}(t),  \dots, \gamma^{(n)}(t)$ generate $\R^n$ for each $t \in [0, 1]$,  where $\gamma^{(k)}$ denotes the $k$-th derivative of $\gamma$.  

At each point $t$ and for each $k$,  one can define the vector space
\[
V_k(t):=\Span(\dot{\gamma}(t),  \dots, \gamma^{(k)}(t)),
\]
which is called the \textit{k-th osculating subspace}.  The full set 
\[
V_0=\{0\} \subset V_1 \subset \dots \subset V_n=\R^n
\]
is called the \textit{osculating flag}.  There is a natural orthonormal basis $e_1(t), \dots, e_n(t)$ defined via
\[
V_j=\Span(e_1,  \dots, e_j)
\]
associated to the flag,  called the Frenet frame.

Restricting to a small neighbourhood and permuting the coordinates,  if necessary,  we can assume that the $i$-th osculating space is not (close to being) contained in the space spanned by the first $i$ coordinates.  For example,  for a given $t_0$,  the normal vector of the $n-1$-osculating space (which is a hyperplane) will have a coordinate which is of size at least $n^{-1/2}$,  and by permuting this coordinate to be the last one (and using that the osculating spaces depend continuously on the curve),  we can modify our setting such that the unit vector $\vec{a}$ spanning $V_{n-1}^\perp$ has $|\vec{a}_n| \gsim 1$,  where the implicit constant only depends on $\gamma$ and $n$.  Here,  we define
\[
W^\perp=\{\vec{v} \in \R^n \Big| \forall w \in W: \langle v, w \rangle=0 \}
\]
for a subspace $W \subset \R^n$,  where $\langle -, - \rangle$ is the standard Euclidean inner product. 

Similarly,  we can find a $c_0$ such that for each $i$, for all $w \in V_i(t)^{\perp}$ with $\Vert w \Vert=1$,
\begin{equation} \label{eq::coordorder}
\Vert \pi_{i} w \Vert \geq c_0,
\end{equation}
where $\pi_i$ is the projection from $\R^n$ to $\R^{n-i}$,  sending $x$ to $(x_{i+1}, \dots, x_n)$.  As the determinant is continuous,  we can also assume that 
\[
|\det(\dot{\gamma}(t),  \dots, \gamma^{(n)}(t))| \geq c_0
\]
for all $0 \leq t \leq 1$.  

Fix $t$ and a normal vector of $V_{n-1}(t)$.  By the choice of the coordinates above,  we can rescale the normal vector $\vec{n}$ such that $(\vec{n})_n=1$ while still $\Vert \vec{n} \Vert \leq c_0^{-1}$.  Write
\[
\underline{v}=(\vec{v}_1, \dots, \vec{v}_{n-1})
\]
for a vector $v \in \R^n$.  

Using projective duality on $\R^n$,  the hyperplanes $V_{n-1}(t)$ correspond to a curve 
\[
\gamma^\ast(t):=(\underline{n},  -\langle \gamma,  \vec{n} \rangle),
\]
called the \textit{dual curve}.  It is a standard result from projective geometry (compare \cite[Theorem 1.1.3]{ovsienko-tabachnikov}) that $(\gamma^{\ast})^\ast=\gamma$ and that the osculating flags of $\gamma$ and $\gamma^\ast$ are dual to each other.  To make this explicit in our notation,  we prove this result here.

\begin{lemma} \label{lem::projdualcurve}
With the notation as above,  we have for all $t$ that
\[
\langle \gamma^{(j)}(t),  \vec{n}^{(i)}(t) \rangle=0
\]
for all $0 \leq i \leq n-1$ and $j \geq 1$ such that $i+j \leq n-1$.  In particular,  this implies that $\vec{n}, \dot{\vec{n}}, \dots, \vec{n}^{(j)}$ span $V_{n-1-j}^\perp$ for any $0 \leq j \leq n-2$.

Moreover,  for any $0 \leq j \leq n-1$,
\[
\frac{d^j}{dt^j} \langle \gamma,  \vec{n} \rangle=\langle \gamma,  \vec{n}^{(j)} \rangle=\langle \underline{\gamma},  \underline{n}^{(j)}(t) \rangle,
\]
and finally,  $(\gamma^*)^\ast=\gamma$. 
\end{lemma}
\begin{proof}
We show the first part by induction over $i$,  with the $i=0$ case following straight from the definition of $\vec{n}$.  So assuming we have shown this claim for a given $i$,  we note that then that for each $j \geq 1$ such that $i+1+j \leq n-1$,
\[
0=\frac{d}{dt} \langle \gamma^{(j)}(t),  \vec{n}^{(i)}(t) \rangle=\langle \gamma^{(j+1)}(t),  \vec{n}^{(i)}(t) \rangle+\langle \gamma^{(j)}(t),  \vec{n}^{(i+1)}(t) \rangle.
\]
As by the induction assumption,  
\[
\langle \gamma^{(j+1)}(t),  \vec{n}^{(i)}(t) \rangle=0,
\]
we find that indeed
\[
\langle \gamma^{(j)}(t),  \vec{n}^{(i+1)}(t) \rangle=0
\]
for all $j \geq 1$ such that $i+1+j \leq n-1$.  

Because $\gamma$ is non-degenerate,  $\dot{\gamma}, \dots, \gamma^{(n)}$ are independent,  which also implies that $\vec{n}, \dots,  \vec{n}^{(n-1)}$ are linearly independent.  As the property above implies that
\[
\vec{n}, \dots,  \vec{n}^{(i)} \in V_{n-1-i}^\perp
\] 
and $V_{n-1-i}^\perp$ has dimension 
\[
n-(n-1-i)=i+1, 
\]
we see that they indeed span  $V_{n-1-i}^\perp$.

We now show that 
\[
\frac{d^j }{dt^j}\langle \gamma,  \vec{n} \rangle=\langle \underline{\gamma},  \vec{n}^{(j)} \rangle
\]
for all $1 \leq j \leq n-1$.  To see this, calculate
\[
\frac{d}{dt} \langle \gamma, \vec{n} \rangle=\langle \dot{\gamma}, \vec{n} \rangle +\langle \gamma, \dot{\vec{n}} \rangle=\langle \gamma,  \dot{\vec{n}} \rangle=\langle \underline{\gamma},  \dot{\underline{n}} \rangle,
\]
where the second equality is because $\langle \dot{\gamma}, \vec{n} \rangle=0$ and the third equality because $(\vec{n})_n=1$,  both by definition of $\vec{n}$.  Iterating the same calculation shows that indeed
\[
\frac{d^j }{dt^j}\langle \gamma,  \vec{n} \rangle=\langle \underline{\gamma},  \underline{n}^{(j)} \rangle
\]
for all $1 \leq j \leq n-1$. 

Finally,  to establish that $(\gamma^\ast)^\ast=\gamma$,  we need to show that
\[
\langle (\underline{\gamma}, 1),  (\gamma^\ast)^{(j)} \rangle=0
\]
for all $1 \leq j \leq n-1$ and that 
\[
\gamma_n=-\langle (\underline{\gamma}, 1),  \gamma^\ast \rangle.
\]
For the latter, note that 
\[
\langle (\underline{\gamma}, 1),  \gamma^\ast \rangle=\langle \underline{\gamma},  \underline{n} \rangle-\langle \gamma,  \vec{n} \rangle=-\gamma_n,
\]
while for the former, by the third assertion of the lemma,
\[
\langle (\underline{\gamma}, 1),  (\gamma^\ast)^{(j)} \rangle=\langle \underline{\gamma},  \underline{n}^{(j)} \rangle+\frac{d^j \gamma^\ast_n}{dt^j}=\langle \underline{\gamma},  \underline{n}^{(j)} \rangle-\langle \underline{\gamma},  \underline{n}^{(j)} \rangle=0,
\]
where we use that for any $j \geq 1$,  $\vec{n}^{(j)}=(\underline{n}^{(j)}, 0)$. Thus,  we see that indeed $(\gamma^\ast)^\ast=\gamma$.  
\end{proof}

\subsection{Transformation of boxes and convex distance functions} \label{subsec::boxtransform}
For $\eps_1 \geq \dots \geq \eps_n$ and an orthonormal basis $v_1, \dots, v_n$ such that $(v_n)_n\geq c_0$ for some absolute $c_0$,  we define the box around $y$ with respect to the basis and these parameters as
\[
B=B(y,  v_i, \eps_i)=\{x \in \R^n: \forall 1 \leq i \leq n \; |\langle x-y, v_i \rangle| \leq \eps_i \}.
\]

Define $H_B:=y+\langle v_1, \dots, v_{n-1} \rangle$ and let $x_B:=x_{H_B}$ be the point in the dual corresponding to $H_B$.  We let $w_i$ for $1 \leq i \leq n-1$ be the vectors in projective space corresponding to rotating a given hyperplane in direction $v_i$; that is,  rotating $\langle v_i, v_n \rangle \cong \R^2$ and leaving all other coordinates fixed.  Concretely,  $(w_i)_n=0$ and $w_i$ is given by
\[
(w_i)_j=\frac{1}{(1+v_i)_n} (\underline{n_{H_B}}+v_i)_j.
\]
We also set $w_n:=(0, \dots, 0,1)$ and $\eps_0:=1$. 

Then, we define the \textit{dual box} $\widetilde{B}$ by
\[
\widetilde{B}=\left\{z \in \R^n \Big| \forall 1 \leq i \leq n-1 :\; |\langle z-x_B,  w_i \rangle| \leq \frac{\eps_n}{\eps_{i}},  \; |\langle z-x_B,  w_n \rangle| \leq \eps_n \right\}.
\]
We also define the scaled dual box by $\widetilde{B}_r:=x_B+r (\widetilde{B}-x_B)$.

\begin{remark}
In the above,  we have used the vector $(0, \dots, 0, 1)$ to induce projective duality.  We could similarly take any other unit vector $v$ with the property that $\langle v,  e_n \rangle \geq c_0$ for example,  by precomposing with a rotation sending $v$ to $(0, \dots, 0, 1)$ and using the established projective duality.
\end{remark}

We let $\gamma$ be a non-degenerate curve in $\R^n$ and let $V_j=V_j(t)$ be its $j$-th osculating subspace at point $\gamma(t)$ (i.e. the subspace spanned by $\gamma^{(i)}(t),  1 \leq i \leq j$).  We now pick bases $v_1,  \dots,  v_n$ at time $t$ of the orthogonal complements $V_j^\perp$ such that $V_{n-j}^\perp=\Span(v_{n-j+1}, \dots, v_n)$ as follows: For fixed $t$,  let $v_n$ be the vector $\vec{n}$ defined in Section \ref{subsec::basics},  that is,  the vector orthogonal to $V_{n-1}$ such that $(v_n)_n=1$;  this is possible by the choice of coordinates (The definition of the vectors $v_i$ technically involves a choice of the ordering of the coordinates - here we take the order $n \rightarrow n-1 \rightarrow n-2 \rightarrow \dots 2 \rightarrow 1$  from above \eqref{eq::coordorder} on \pageref{eq::coordorder}).  

Similarly,  we let $v_{n-1}$ be the vector such that 
\[
\Span(v_{n-1},  v_n)=V_{n-2}^\perp 
\]
and $(v_{n-1})_n=0$,  $(v_{n-1})_{n-1}=1$,  which is again possible by the way we chose the coordinates.  For general $1 \leq i \leq n$,  we let $v_i$ be the vector such that
\[
\Span(v_{i},  \dots, v_n)=V_{i-1}^\perp,
\]
$(v_i)_i=1$ and $(v_i)_j=0$ for $j>i$.  It is possible to pick such vectors by the way in which we arranged the coordinates earlier,  and moreover,  we have $\Vert v_i \Vert \lsim 1$,  with the implicit constant only depending on the dimension $n$ and the constant $c_0$ from above.  

\begin{lemma} \label{lem::vjrelation}
With this definition,  for each $t$ and each $1 \leq i \leq n-1$,
\[
v_{i}(t)=\frac{1}{(\dot{v}_{i+1}(t))_{i}} \dot{v}_{i+1}(t).
\]
\end{lemma}
\begin{proof}
It is obvious that the choice of $v_i$ above is unique,  so if we can show that $\dot{v}_{i+1}$ has the property that $\dot{v}_{i+1},  v_{i+1}, \dots, v_n$ span $V_{i-1}^\perp$ and $(\dot{v_{i+1}})_j=0$ for $j>i$,  we are done.  

For the latter property,  note that as $(v_{i+1})_j=\delta_{j=i+1}$,  
\[
\frac{d(v_{i+1})_j}{dt}=0
\]
for $j=i+1, \dots, n$.  For the former property,  write
\[
v_{i+1}(t)=\sum_{j=i+1}^{n} a_j(t) \vec{n}^{(n-j)}(t),
\]
which is possible because by induction,  $v_{i+1}(t) \in V_{i}^\perp(t)$,  while by  Lemma \ref{lem::projdualcurve},  $V_{i}^\perp(t)$ is spanned by $\vec{n}^{(l)}(t)$ with $l=0, \dots, n-1-i$.  Moreover, because $v_{i+1}(t), \dots, v_n(t)$ span $V_{i}^\perp(t)$,  we must have that $a_{i+1}(t) \neq 0$. Because the coefficient functions $a_j$ are implicitly defined through differentiable equations and we have restricted the interval to be away from where they become singular,  the functions $a_j$ are differentiable (and bounded in size). 

Differentiating,  we find
\[
\dot{v}_{i+1}(t)=\sum_{j=i+1}^{n} \dot{a}_j(t) \vec{n}^{n-j}(t)+a_j(t) \vec{n}^{n-j+1}(t)=a_{i+1}(t)\vec{n}^{n-i}(t)+b(t),
\]
where $b(t) \in V_{i}^\perp(t)$.  As $a_{i+1}(t) \neq 0$ as mentioned above,  we see that indeed
\[
\langle \dot{v}_{i+1}(t) \rangle_\R+V_{i}^\perp(t)=V_{i-1}^\perp(t).
\]
\end{proof}

We then define vectors $w_0=(1, 0, \dots, 0) \in \R^{n+1}$ and
\[
w_j=(-\langle \gamma, v_j \rangle,  v_j) \in \R^{n+1}.
\]

As an example,  we then have that $w_{n}=(\gamma_n^\ast(t),  \underline{\gamma}^\ast(t), 1)$.  As $v_j$ and $\gamma$ depend on $t$,  so do the $w_j$. 

Combining Lemma \ref{lem::vjrelation} and Lemma \ref{lem::projdualcurve},  we get that for each $1 \leq i \leq n-1$,
\begin{equation} \label{eq::wjrelation}
w_{i}(t)=\frac{1}{(\dot{w}_{i+1}(t))_{i}} \dot{w}_{i+1}(t).
\end{equation}
Indeed,  in all but the $0$-th coordinate,  this is exactly Lemma \ref{lem::vjrelation},  while 
\[
\frac{d}{dt}\langle \gamma, v_i \rangle=\langle \gamma, \dot{v}_i \rangle
\]
because
\[
\langle \dot{\gamma}, v_i \rangle=\sum_{j=i}^{n} a_j(t) \langle \dot{\gamma},\dot{\vec{n}}^{(n-j)}(t)=0
\]
 for all $2 \leq i \leq n$ by Lemma \ref{lem::projdualcurve}.

Define the matrix
\[
A=A(t):=(w_0, \dots, w_n),
\]
which is upper triangular.  

Now, define a different sets of vectors $\widetilde{w_j}$ on $\R^{n+1}$ (indexed by $0,  1, \dots, n$) as follows.  Set 
\[
\widetilde{w_n}:=(1,  \gamma),
\]
with the $1$ in the $0$-th coordinate, and for $1 \leq i <n$,
\[
\widetilde{w_i}:=\frac{1}{a_i(t)} \frac{d \widetilde{w_{i+1}}}{dt},
\]
where $a_i(t)=\left(\frac{d \widetilde{w_{i+1}}}{dt} \right)_{n-i}$ is such that $(\widetilde{w_{i}})_{n-i}=1$ (replicating the structure we saw in Lemma \ref{lem::vjrelation}).  We also set $\widetilde{w_0}=(0, \dots, 0, 1)$.  

This makes the matrix 
\[
\widetilde{A}:=(\widetilde{w_n},  \widetilde{w_{n-1}},  \dots, \widetilde{w_1}, \widetilde{w_0})
\]
lower triangular.  

\begin{lemma} \label{lem::wj_n}
Let $\widetilde{v_j}$ be the vectors on $\R^n$ such that $(\widetilde{v_j})_i=(\widetilde{w_j})_i$ for all $0 \leq i \leq n-1$.  For any $1 \leq j \leq n-1$,  the vector $v_{n-j}$ can be written as a linear combination of $\underline{\dot{\gamma}}, \dots,  \underline{\gamma}^{(j)}$.

Moreover, for any $1 \leq j \leq n$,
\begin{equation} \label{eq::gammatilde}
(\widetilde{w_j})_n=-\langle \widetilde{\gamma},  \widetilde{v_j} \rangle,
\end{equation}
where  
\[
\widetilde{\gamma}:=(-\langle \gamma, \vec{n} \rangle, \underline{n})
\]
(which is up to permutation of the coordinates the same as the dual curve $\gamma^\ast$). 
\end{lemma}
\begin{proof}
The first assertion is obvious once we recall that $v_n=(1, \underline{\gamma})$ and use the iterative definition of the $w_j$.

For the second assertion,  note that by Lemma \ref{lem::projdualcurve},
 \[
\dot{\widetilde{\gamma}}=(-\langle \underline{\gamma},  \dot{\underline{n}} \rangle,  \dot{\underline{n}}),
 \]
 so in particular,
 \[
 \langle \dot{\widetilde{\gamma}}, v_n \rangle=-\langle \underline{\gamma},  \dot{\underline{n}} \rangle+\langle \underline{\gamma},  \dot{\underline{n}} \rangle=0.
 \]
By the first assertion and the first statement in Lemma \ref{lem::projdualcurve}, we also have that
 \[
  \langle \dot{\widetilde{\gamma}}, v_j \rangle=0
 \]
 for all $2 \leq j \leq n-1$.  From this,  we see that
 \begin{equation} \label{eq::wj_n}
 \frac{d}{dt} \langle \widetilde{\gamma}, \widetilde{v_j} \rangle=\langle \widetilde{\gamma}, \dot{\widetilde{v_j}} \rangle
 \end{equation}
 for all $2 \leq j \leq n-1$.  
 
To see the second assertion from this,  recall that we have seen
\[
(\widetilde{w_j})_n=\gamma_n=-\langle (\underline{\gamma}, 1), \gamma^\ast \rangle=-\langle \widetilde{\gamma},  \widetilde{v_n} \rangle
\]
 in the proof of Lemma \ref{lem::projdualcurve},  so the claim is true for $j=n$.  For smaller $j$,  it is seen by using \eqref{eq::wj_n} while inducting down.
\end{proof}
From this,  we can see the next lemma showing that this definition leads to the dual matrix.
\begin{lemma} \label{lem::dualmatrix}
With these definitions,  $\widetilde{A}=\left(A^{-1}\right)^T$; or equivalently,  for any $0 \leq i, j \leq n$,  
\[
\langle w_i,  \widetilde{w_j} \rangle=\delta_{i+j=n}.
\]
\end{lemma}
\begin{proof}
Because the matrices $\widetilde{A}^T$ and $A$ are upper triangular with ones on the diagonal, it suffices to show that 
\[
(\widetilde{A}^T \cdot A)_{ij}=\langle w_i,  \widetilde{w_j} \rangle=0
\]
for $i$ and $j$ such that $i+j>n$.  We split into several cases. 

For $j=n$ and $i=1, \dots,  n$, we note that  
\[
\langle w_i,  \widetilde{w_n} \rangle=-\langle \gamma,  v_i \rangle+\langle v_i, \gamma \rangle=0
\]
by definition of $w_i$.  

For the  $i=n$ and $j=1, \dots, n-1$ case,  we recall that $w_n=(\ast, \underline{n}, 1)$, so this follows from the  second part of Lemma \ref{lem::wj_n}.

Finally,  for $1 \leq i, j \leq n-1$ such that $i+j>n$,  because $(w_i)_n=0$ and $(\widetilde{w_j})_0=0$,  we have that 
\[
\langle w_i,  \widetilde{w_j} \rangle=\langle v_i^\prime,  \widetilde{v_j}^\prime \rangle,
\]
where $v_i^\prime$ and $\widetilde{v_j}^\prime$ are the projection of $w_i$ and $\widetilde{w_j}$ to the space spanned by the coordinates $1, \dots, n-1$, respectively. 

By Lemma \ref{lem::vjrelation},  any $v_i^\prime$ can be written as a linear combination of $\dot{\underline{n}},  \dots, \underline{n}^{(n-i)}$ , while by the first assertion of Lemma \ref{lem::wj_n},  any $\widetilde{v_j}^\prime$ can be written as a linear combination of $\dot{\underline{\gamma}},  \dots, \underline{\gamma}^{(n-j)}$.  By Lemma \ref{lem::projdualcurve},
\[
\langle \vec{n}^{(n-i)}, \gamma^{(n-j)} \rangle=0
\]
for $i, j \leq n-1$ and $i+j>n$; and moreover, for any $i \leq n-1$, 
\[
\langle \vec{n}^{(n-i)}, \gamma^{(n-j)} \rangle=\langle \underline{n}^{(n-i)}, \underline{\gamma}^{(n-j)} \rangle
\]
because $(\vec{n}^{(l)})_n=0$ for all $l \geq 1$. Taken together,  this shows the claim in the case that $1 \leq i, j \leq n-1$ and $i+j>n$.
\end{proof}

\section{Duality principle (I): Boxes} \label{sec::dualboxes}
In this section,  we show a purely volume based duality result (Lemma \ref{lem::boxdual}) and generalise this to a duality result about the number of rational points in each of the boxes.  We end up with a satisfying result detailing that in the dense setting,  any deviation from the expected number of rational points is due to rational points of small denominator being close to the dual.  Here, with small denominator we mean that the maximal denominator $\widetilde{Q}$ is small enough to be in the sparse setting.  Estimating how often this can happen in the sparse setting can then be controlled via the determinant method.  

The terminology of `dual box' is justified by the lemma below. 
\begin{lemma} \label{lem::boxdual}
There is an $N$ and constants $c, c_0,  C$ depending only on the dimension and unit vectors $u_1,  \dots, u_N$ such that the following holds.  

For any hyperplane $H$, if there is a $u_j$ with $\langle u_j,  v_n \rangle \geq c_0$ and the property that $x_H$ (the to $H$ associated point under projective duality with respect to $u_j$) lies in $\widetilde{B}_c$,  then
\[
\vol(H \cap B) \gsim \eps_1 \dots \eps_{n-1}.
\]
Conversely,  if 
\[
\vol(H \cap B) \geq \frac{C}{A} \eps_1 \dots \eps_{n-1}
\]
for some $A$,  then there is a vector $u_j$ such that under projective duality induced via $u_j$,  we have that
\[
x_H \in \widetilde{B}_A \cap B(1),
\]
where $B(1)$ is the ball of radius one around the origin in dual space.
\end{lemma} 
\begin{proof}
Let $c_0>0$ be small and let $u_1, \dots, u_N$ be a $c_0$-net of $\mathbb{S}^{n-1}$.  For a given box $B$ and hyperplane $H$,  we can pick $u_j$ with $\langle u_j, v_n \rangle \geq c_0$ and $\langle u_j, \vec{n_H} \rangle \geq c_0$ (replacing $\vec{n_H}$ by $-\vec{n_H}$ if necessary),  and we use projective duality with respect to this $u_j$ (after rotating $u_j$ to $(0, \dots, 0, 1)$).  Writing $x_H=x_B+\sum_{j=1}^{n} \theta_j w_j$,  this gives $|\theta_j| \lsim 1$ for all $j$, where all implicit constants depend on $c_0$.  The argument below parametrises $H \cap B$ over $\Span(v_1, \dots, v_{n-1})$,  which requires $\langle \vec{n_H}, v_n \rangle \gsim 1$.  If this fails,  the same bounds follow because a hyperplane through the centre of a box cuts it in a section whose area is comparable to the volume of the box divided by its width in the direction normal to the hyperplane.

Without loss of generality,  the centre $y$ of $B$ is the origin. We first assume that $y \in H$ and write $x_H=x_B+\sum_{j=1}^{n-1} \theta_j w_j$.  We then have that
\[
\pi_{n-1} \vec{n_H}=\sum_{j=1}^{n-1} a_j v_j
\]
with $a_j \sim \theta_j$.  Because all $\theta_j$ are bounded by a small constant,  we have that $\langle \vec{n_H}, v_n \rangle \sim 1$.  

We let $\tau : V_{n-1} \to H$ be the inverse of $\pi_{n-1} : H \to V_{n-1}$ and observe that then for each vector $v=\sum_{j=1}^{n-1} t_j v_j \in V_{n-1}$, because $\langle \vec{n_H}, v_n \rangle v_n=\vec{n_H}-\sum_{j=1}^{n-1} a_j v_j$, 
\[
|\langle \tau(v),  v_n \rangle|=\frac{1}{\langle \vec{n_H}, v_n \rangle} \left|\sum_{j=1}^{n-1} \langle \tau(v),   a_j v_j \rangle \right|=\frac{1}{\langle \vec{n_H}, v_n \rangle}   \left|\sum_{j=1}^{n-1} \langle v,   a_j v_j \rangle \right| \sim \left|\sum_{j=1}^{n-1} a_j t_j \right|.
\]

For the first statement,  as long as $t_j \leq \eps_j$ for all $j$,  we have that
\[
 \left|\sum_{j=1}^{n-1} a_j t_j \right| \lsim \sum_{j=1}^{n-1}|\theta_j| \eps_j \lsim \eps_n,
 \]
 which means that as long as $t_j \leq c \eps_j$ for some small but absolute constant $c$,  the corresponding point on the plane $H$ lies in $B$.  As long as $\theta_n \leq c \eps_n$ for some small constant $c$,  its contribution in this calculation is negligible,  which implies the claim.  
 
For the second part of the lemma,  we show the following more precise statement,  which also keeps track of translations of $H$.  Let $H$ be a hyperplane through $y$,  let
\[
A:=\max_{1 \leq i \leq n-1} |\theta_i| \frac{\eps_i}{\eps_n},
\]
and for $r \in \R$ let $H_r:=H+r \vec{n_H}$.  Then there are constants $c_1, c_2$ such that
\[
\vol(H_r \cap B) \sim \frac{1}{A} \eps_1 \dots \eps_{n-1} \quad \textrm{if } |r| \leq c_1 A \eps_n,
\qquad
\vol(H_r \cap B) \lsim \frac{1}{A} \eps_1 \dots \eps_{n-1} \quad \textrm{for all } r,
\]
and $\vol(H_r \cap B)=0$ if $|r| \geq c_2 A \eps_n$,  where the volume is taken with respect to the surface measure on $H_r$.  

We again start by assuming that $r=0$.  By the calculation above,
\[
|\langle \tau(v),  v_n \rangle| \sim \left|\sum_{j=1}^{n-1} a_j t_j \right|
\]
and by the choice of $A$ we have
\[
\left|\sum_{j=1}^{n-1} a_j t_j \right|=\eps_n A \left|\sum_{j=1}^{n-1} \frac{\eta_j t_j}{\eps_j} \right|
\]
with $|\eta_j| \leq 1$ and $\eta_i=1$ for some index $i$.  The lower bound follows as in the previous case,  and for the upper bound, we find
\begin{align*}
\vol(H \cap B) &\sim \vol\left(\vec{t} \in \R^{n-1} \Bigg| \left|\sum_{j=1}^{n-1} \frac{\eta_j t_j}{\eps_j} \right| \leq A^{-1} \mathrm{\; and \;} \forall 1 \leq i \leq n-1, \; |t_i| \leq \eps_i \right)\\
&=\eps_1 \dots \eps_{n-1} \vol\left(\vec{t} \in [-1, 1]^{n-1} : \left|t_i + \sum_{j \neq i} \eta_j t_j \right| \leq A^{-1}\right) \\
&\lsim \frac{1}{A} \eps_1 \dots \eps_{n-1}.
\end{align*}

Finally,  when $r \neq 0$, we get an extra constant term of $r$ in the above sums, which changes the expression to
\[
|\langle \tau(v),  v_n \rangle| \sim \eps_n A \left|\frac{r}{\eps_n A}+\sum_{j=1}^{n-1} \frac{\eta_j t_j}{\eps_j} \right|.
\]

From this,  it is clear that 
\[
\vol(H \cap B) \sim |\langle \tau(v),  v_n \rangle| \lsim \frac{1}{A} \eps_1 \dots \eps_{n-1}
\]
for all $r$.  Moreover,  for $c_1$ small enough compared to the implicit constant,  the conclusion of $\vol(H \cap B) \sim \frac{1}{A} \eps_1 \dots \eps_{n-1}$ still holds for $|r| \leq c_1 A \eps_n$,  and for $c_2$ sufficiently big compared to the implicit constant,  for any $r \geq c_2 A \eps_n$ the set of $\vec{t}$ with $|t_j| \leq \eps_j$ for all $j$ satisfying the condition is empty.  This concludes the proof. 
\end{proof}

\begin{remark}
The lemma is stated for orthonormal bases,  but it is really a duality statement for the flag $V_1 \subset V_2 \subset \dots \subset V_n$.  As such,  it naturally carries over to all bases of $V_i$ that are not `too parallel': Indeed,  if $u_1, \dots, u_n$ are vectors in $\R^n$ such that for any $1 \leq i \leq n$,  
\[
\Vert u_{i+1} - \pi_i u_{i+1} \Vert_{\R^n} \geq c_0
\]
for some absolute $c_0$,  where $\pi_i$ is the projection from $\R^n$ onto the subspace $U_i=\Span(u_1, \dots, u_i)$,  we can define the orthonormal basis by the requirement that $v_1, \dots, v_j$ should be an orthonormal basis of $U_j$ for all $j$ (so $V_j=U_j$ in our previous notation).  Because of the condition $\eps_1 \geq \dots \geq \eps_n$ we then have that an analogous result to Lemma \ref{lem::boxdual} holds for the parallelepiped
\[
P=P(y,  u_i, \eps_i)=\{x \in \R^n \Big| \forall {1 \leq i \leq n}: \; |\rho_i(x-y)| \leq \eps_i \},
\]
where $\rho_i(x)$ is defined via $x=\sum_{i=1}^n \rho_i(x) u_i$.  This is true because by the assumption $\eps_1 \geq \dots \geq \eps_n$,
\[
P \sim B:=\{x \in \R^n \Big| \forall {1 \leq i \leq n}: \; |\langle x-y,  v_i \rangle| \leq \eps_i \},
\]
where  $P \sim B$ is to be understood in the sense that there are $c_1,  c_2$ (depending on $c_0$ and the ambient dimension) such that $c_1 (P-y) \subset B-y \subset c_2 (P-y)$.  The previous result can then be applied to the box $B$ spanned by the orthonormal basis $v_1, \dots, v_n$.  
\end{remark}

The most powerful generalisation of Lemma \ref{lem::boxdual} is Theorem \ref{thm::maindualitybox} from the introduction,  which is our central duality result.  We restate it here for convenience.
\begin{restated}{Theorem \ref{thm::maindualitybox}}
Let $n$ be an integer.  There are constants $c_0, C>0$,  an integer $N$ and integer matrices $\Phi_1, \dots, \Phi_N \in SL_n(\Z)$ with orthogonal columns, all depending only on the dimension (picking $c_0$ first,  then $N$,  and then $C$),  such that the following holds.  For any $x_0 \in \R^n$,  for any orthonormal basis $e_1,  \dots, e_n$ and for any $c_0 \geq \eps_1 \geq \dots \geq \eps_n >0$ and $Q \in \mathbb{N}$ such that 
\[
\eps_1 \cdot \dots \cdot \eps_n \geq C Q^{-(n+1)},
\]
consider the box 
\[
V:=\{x_0+\sum_{j=1}^n \beta_j e_j,  |\beta_j| \leq \eps_j \forall 1 \leq j \leq n \}.
\]
Set $ \widetilde{\eps_j}:=\frac{\eps_n}{\eps_{n-j}}$ with $\eps_0:=1$ for $1 \leq j \leq n$ and $\widetilde{Q}:=\frac{1}{\eps_n Q}$.

For any parameter $A \geq 1$,  if
\[
\mathcal{N}(Q,  V) \geq A \eps_1 \dots \eps_n Q^{n+1},
\]
the rational points in $V$ up to height $Q$ must then lie in a rational subspace of dimension $n-r$ for some $r \geq 1$.  There is then a $\Phi=\Phi_j$ from the above finite collection with $(\Phi e_n)_n \geq c_0$ such that for $V_\Phi:=x_0+\Phi (V-x_0)$, the dual box
\[
\widetilde{V}_\Phi=\left\{x_{V_\Phi}+\sum_{j=1}^n \beta_j \widetilde{\Phi e_j},  |\beta_j| \leq \widetilde{\eps_j}\right\},
\]
where $x_{V_\Phi}$ is the dual point of the hyperplane orthogonal to $\Phi e_n$ passing through $x_0$,  satisfies the following.

For $1 \leq i \leq r$, there are $B_i \geq 1$, $m_i \leq \widetilde{Q}/B_i$ such that
\[
\prod_{i=1}^r \frac{\widetilde{Q}}{B_i m_i} \gsim A
\]
and there are rational points $\vec{a_i}/m_i \in (\widetilde{V}_\Phi)_{B_i} \cap K(C)$,  where $\vec{a_i} \in \Z^n$ and $K(C)$ is the compact ball around the origin with radius $C$. The rational points $\vec{a_i}/m_i$ are in general position,  meaning that no $s+2$ of them are contained in an $s$-dimensional affine subspace.  The rational subspace of dimension $n-r$ above is $H_1 \cap \dots \cap H_r$,  where $H_i$ is the hyperplane corresponding to $\vec{a_i}/m_i$ under projective duality.

Moreover,  if the box is such that $\vec{a}/m \notin (\widetilde{V}_\Phi)_{C \widetilde{Q}/m} \cap K(C)$ for all $\Phi=\Phi_j$ and all integers $(m,  \vec{a}) \in \Z^{n+1}$,  then 
\[
\mathcal{N}(Q,  V) \gsim \eps_1 \dots \eps_n Q^{n+1}.
\]
\end{restated}
\begin{remark}
We have that
\[
\widetilde{\eps_1} \dots \widetilde{\eps_n}=\frac{\eps_n^{n+1}}{\eps_1 \dots \eps_n} \leq  \eps_n^{n+1} Q^{n+1}=\widetilde{Q}^{-(n+1)}, 
\]
so we are indeed in the sparse setting. 
\end{remark}
\begin{remark} \label{rem::Phiexplanation}
The presence of the $N$ many integer matrices $\Phi_j$ might seem mysterious,  but the necessity of their presence and the way in which they are picked can be explained geometrically.  As there are only boundedly many,  they are of course completely irrelevant for all practical applications of the theorem. 

To explain where the $\Phi_j$ are coming from,  let $L$ be the $n-r$-dimensional subspace all rationals up to height $Q$ in $V$ lie on.  If we assume for the moment that there are points as in the conclusion of the theorem,  corresponding to the successive minima of the dual convex distance function as in the proof of Theorem \ref{thm::maindualitybox},  then the associated rational planes $H_1,  \dots,  H_r$ have respective height $m_i$, $\vol(H_i \cap V) \sim \frac{1}{B_i} \eps_1 \dots \eps_{n-1}$ and satisfy
\[
H_1 \cap \dots \cap H_r=L,
\]
where 
\[
H_i=\{x \in \R^n: \langle (\underline{a_i}/m_i,  
1),  x \rangle=(\vec{a_i})_1/m_i \}.
\]

A priori,  knowing only $L$,  we could reverse engineer these hyperplanes by letting $H_1$ be the hyperplane maximizing $\vol(H_i \cap V)/\mathrm{height}(H)$ among all hyperplanes passing through $L$.  Using a similar maximizing procedure for the rest of the $H_i$ (among all rational hyperplanes that lead to $H_1 \cap \dots \cap H_{i-1} \cap H_i$ being $n-i$ dimensional),  we can define the associated hyperplanes in physical space without reference to any coordinate frame. 

The problems with inducing projective duality occur now if the normal vector $\vec{n}_{H_i}$ of one of the $H_i$ has $(\vec{n}_{H_i})_n=0$ or is very small,  for example because the rational space $L$ lies in a coordinate line.  There is no way to fix this without introducing some change in the coordinate frame like $\Phi$ to force $|(\vec{n}_{H_i})_n| \geq c_0$ for some small absolute $c_0$.  However,  by picking a collection of matrices $\Phi_j$ such that their normalised columns form a $c_0/2$-dense subset in the space of all orthornormal bases of $\R^n$, we can assure that there is always a set of coordinates $\Phi_j$ such that $|(\vec{n}_{\Phi H_i})_n| \geq c_0$ (and also $|(\Phi e_j)_n| \geq c_0$ for all $j$).  This motivates the definition of $\Phi_j$ in the below.  
\end{remark}
\begin{proof}[Proof of Theorem \ref{thm::maindualitybox}]
Let $c_0$ be a small constant.  As mentioned in Remark \ref{rem::Phiexplanation},  we let $\Phi_j \in SL_n(\Z)$ be integer matrices such that all its columns are orthogonal,  and such that the normalised columns are $c_0/2$-dense in the space of all orthogonal bases in $\R^n$.  We will consider the boxes $\Phi V$ with orthonormal bases $\{(\Phi_j e_i)_i,  j\}$ later,  but for now will just consider $e_j$. 

Define new vectors $e_j^\prime$ iteratively down from $j=n$ via the requirement that $(e_j^\prime)_j=1$,  $(e_j^\prime)_i=0$ for $i>j$ and that $e_j^\prime \in \Span(e_n,  \dots,  e_j)$.  By ordering the coordinates in the right way,  we can assure that the size of the $e_j^\prime$ remains bounded.  Let $v_0=(1, 0 \dots, 0)$ and $v_i=(-\langle x_0,  e^\prime_i \rangle,  e_i^\prime)$ and note that the $v_i$ then form an upper triangular matrix.

Now,  consider the convex distance function $F : \R^{n+1} \to \R_+$ given by
\[
F(q,  \vec{a}):=\max_{i=0, \dots, n} \left( \frac{1}{\eps_i Q}|\langle(q, \vec{a}),  v_i \rangle|  \right),
\]
where $\eps_0:=1$,  which has 
\[
\vol(\{x: F(x) \leq 1 \}) \sim \eps_1 \dots \eps_n Q^{n+1} \geq 1.
\]
A quick calculation shows that then
\[
\mathcal{N}(V, Q) \lsim \#\{(q, \vec{a}) \in \Z^{n+1}: F(q,  \vec{a}) \leq 1 \}.
\]
Let $\lambda_j$ be the successive minima of $F$ with respect to the lattice $\Z^{n+1}$; by Theorem \ref{thm::polarcount},  we know that 
\[
\#\{(q, \vec{a}) \in \Z^{n+1}: F(q,  \vec{a}) \leq 1 \} \sim \vol(\{x: F(x) \leq 1 \}) \prod_{j=1}^n \max \left(\frac{1}{\lambda^*_j},  1 \right),
\]
so under our assumption,  
\[
A C \eps_1 \dots \eps_n Q^{n+1} \leq \mathcal{N}(V, Q) \lsim \eps_1 \dots \eps_n Q^{n+1} \prod_{j=1}^n \max \left(\frac{1}{\lambda^*_j},  1 \right).
\]
Hence,  there is an $r<n$ and vectors $(\vec{a_i},  m_i) \in \Z^{n+1}$ not contained in a lower dimensional subspace such that 
\[
\prod_{i=1}^r (F^*(\vec{a_i},  m_i))^{-1} \sim \prod_{i=1}^r (\lambda^*_i)^{-1} \gsim C A.
\]

By Lemma \ref{lem::polardist} and by the same computation as in Lemma \ref{lem::dualmatrix}, 
\[
F^*(x) \sim \max \left( \frac{1}{\widetilde{\eps_i} \widetilde{Q}}|\langle x,  \widetilde{v_{n-i}} \rangle|  \right),
\]
where $\widetilde{v_{n-i}}$ forms a lower triangular matrix that is the adjoint inverse of the matrix $v_1,  \dots, v_n$ and we let the final coordinate be the one in which we have the term with $\widetilde{v_n}=(0, \dots, 0, 1)$.   We have also seen that $\widetilde{v_{n-i}}$ has the same form as $v_i$,  only this time with the dual box.  Hence,  if we knew that $m_i \neq 0$ and $|\vec{a_i}|/m_i \leq C$ for all $i$,  we could let 
\[
B_i:= \frac{\widetilde{Q}}{m_i} F^*(\vec{a_i},  m_i)
\]
and would have a well-defined rational point in the dilated dual box as required.  

This might not be true for $e_j$ for the reasons mentioned in Remark \ref{rem::Phiexplanation} but we can be sure that after the same procedure, at least one of the $\Phi_j$ as chosen as earlier has this property.  Indeed,  if we let $\phi$ be the last column of each $\Phi$,  we know that we can choose it among a $c_0/2$ dense set of directions.  By the way the problem transforms and because the matrix $\mathrm{Id}_1 \otimes \Phi \in SL_{n+1}(\Z)$ keeps the integer lattice invariant,  we have that the new vectors $(\vec{a}_i^\prime,  m_i^\prime)$ still correspond to the successive minima of the transformed convex distance function $F^*$,  and we have that $(\vec{a}_i^\prime,  m_i^\prime)=(\mathrm{Id}_1 \otimes \Phi)^T (\vec{a}_i,  m_i)$.  Writing $\underline{a_i}$ for $((a_i)_2, \dots, (a_i)_n)$,  we then have
\[
m_i^\prime=\langle \phi,  (\underline{a_i}, m_i) \rangle.
\]
As the factor in front of the first coordinate of $F^*$ is $Q \geq 1$,  the vector $(\underline{a_i}, m_i)$ must be non-zero for $F^*(\vec{a_i},  m_i)<1$ to hold.  But now,  we are pretty much done by the choice of the $\Phi$:  The vectors $(\underline{a_i}, m_i)$ lie in at most $n$ many distinct directions and the set of $\{\phi \in \mathbb{S}^{n-1}: \langle \varphi,  (\underline{a_i}, m_i) \rangle \leq c_0\}$ makes up a set of width at most $c_0$ on the great circle perpendicular to them,  so for $c_0$ small enough,  by volume considerations there is some unit vector $\varphi$ in $\R^n$ with $|\langle \varphi,  (\underline{a_i}, m_i) \rangle| \geq c_0 \Vert (\underline{a_i}, m_i) \Vert$ for all $i$.  Because our collection of $\phi$ is $c_0/2$-dense in the space of directions,  there is then also one of the vectors in our collection with $|\langle \phi,  \vec{a_i} \rangle| \geq c_0/2  \Vert a_i \Vert$ for all $i$.  Choosing the $\Phi$ associated to this $\phi$,  the new points $\vec{a}_i^\prime/m_i^\prime$ then indeed lie in a bounded set around the origin.  Moreover,  for every integer vector $x$ with $F(x) \leq 1$,  the integer $\langle x,  (\vec{a}_i, m_i) \rangle$ satisfies $|\langle x,  (\vec{a}_i, m_i) \rangle| \leq F(x) F^*(\vec{a}_i, m_i)<1$ and hence vanishes,  so all rational points in $V$ lie in $H_1 \cap \dots \cap H_r$.  

For the final part, we know by our assumption that then $F^*(\vec{m}) \geq C^\prime$ for some large $C^\prime$ and all non-zero integer vectors $\vec{m}$. Hence, by Mahler's theorem, we know that $F(\vec{m}) \leq c$ for some small constant $c$. For $c$ small enough, the lattice vectors tile $\{F \leq 1\}$ and thus have at least the expected count, similarly to the argument in Theorem \ref{thm::polarcount}. 
\end{proof}

\section{Duality principle (II): The neighbourhood of manifolds} \label{sec::dualcurves}
We let $\gamma: [0,1] \to \R^n$ be a $n$-times differentiable non-degenerate curve.  We let $V_1 \subset V_2 \subset \dots \subset V_n=\R^n$ be its osculating flag, that is, $V_j=\Span(\gamma^{(i)}(t),  1 \leq i \leq j)$.  We let $v_1,  \dots, v_n$ be the orthonormal basis adapted to the osculating flag.  For $\eps_1 \geq \dots \geq \eps_n$,  we define the $\eps_1 \times \dots \times \eps_n$\textit{-box adapted to the curve at time} $t$ as
\[
B_{\mathbf \eps}(t)=\{x \in \R^n \Big| \forall 1 \leq i \leq n : \; |\langle x-\gamma(t),  v_i \rangle| \leq \eps_i \},
\]
where we use the subscript $\mathbf \eps$ to denote the dependence on the vector $\mathbf \eps:=(\eps_1,  \dots,  \eps_n)$. 
We again set $\eps_0:=1$ and define  
\[
\Delta:=\min \left\{ (\eps^{-1}_j \eps_k)^{\frac{1}{k-j}},  0 \leq j<k \leq n \right\}
\]
to be the \textit{step size} associated to the configuration $\mathbf \eps$.  We define the $\eps_1 \times \dots \times \eps_n$\textit{-adapted neighbourhood of the curve} $\gamma$ to be 
\[
V(\mathbf{\eps})=\bigcup_{t \in [0,1]} B_{\mathbf \eps}(t).
\]
The lemma below shows that the boxes around the set $t_j$ are close to disjoint and essentially cover $V(\mathbf \eps)$.
\begin{lemma} \label{lem::minimalcover}
There are constants $c,  C$ depending only on the dimension and the curvature bounds for $\gamma$,  such that the following holds: Let $t_1, \dots, t_N$ be $N=\left \lfloor \Delta^{-1} \right\rfloor$ many $\Delta$-separated points in $[0,1]$.  Then,
\begin{enumerate}
\item $V(\mathbf \eps) \subset \bigcup_{j=1}^N B_{C\mathbf \eps}(t_j)$;
\item If the minimum in $\Delta$ is attained for one of the terms with $j=0$,  then
\[
B_{c\mathbf \eps}(s) \cap B_{c\mathbf \eps}(t)=\emptyset
\]
for all $s, t$ with $|s-t| \geq \Delta$.
\item If one of the terms $(\eps_j^{-1} \eps_k)^{1/(k-j)}$ with $1 \leq j<k$ attains the minimum in $\Delta$,  the neighbourhoods $B_{\mathbf \eps}(t)$ are still separated in the sense that  
\[
\vol(B_{\mathbf \eps}(s) \cap B_{\mathbf \eps}(t)) \lsim \left(\frac{\Delta}{|s-t|} \right)^{k-j} \eps_1 \dots \eps_n.
\]
\end{enumerate}
Here,  for a constant $a>0$,  $B_{a\mathbf \eps}(t)$ denotes the box associated to the configuration $a \eps_1,  \dots,  a \eps_n$.  
\end{lemma} 
\begin{proof}
Let $e_1(t),  \dots, e_n(t)$ be the Frenet frame for the curve.  We notice that for the first assertion,  it suffices to show that there is a $C$ such that for any $t, s$ with $|t-s| \leq \Delta$,  we have that
\[
B_{\mathbf \eps}(s) \subset B_{C \mathbf \eps} (t).
\]

To do so, we note that by Lemma \ref{lem::frenetchange},  
\[
|\langle \gamma(s)-\gamma(t), e_j(t) \rangle| \sim |s-t|^j
\]
while 
\[
|\langle e_i(s), e_j(t) \rangle| \lsim |t-s|^{|i-j|}
\]
for $i \neq j$.  Thus, any point $y \in B_{\mathbf \eps}(s)$ satisfies
\begin{align*}
|\langle y-\gamma(t),  e_j(t) \rangle| \leq |\langle \gamma(s)-\gamma(t),  e_j(t) \rangle| + \sum_{j=1}^n \eps_i |\langle e_i(s),  e_j(t) \rangle| \lsim |s-t|^j+\sum_{i =1}^n \eps_i |s-t|^{|j-i|}.
\end{align*}
Now,  by definition of $\Delta$ and because $|s-t| \leq \Delta$,
\[
|s-t|^j \leq \Delta^j \leq \eps_j,
\]
while because $\eps_i \geq \eps_l$ for $i \leq l$ and again by definition of $\Delta$,
\[
\sum_{i =1}^n \eps_i |s-t|^{|j-i|} \lsim \sum_{i=1}^j \eps_i  |s-t|^{j-i} \leq \sum_{i=1}^j \eps_i  \Delta^{i-j} \leq \sum_{i=1}^j \eps_i  (\eps_i^{-1} \eps_j) \lsim \eps_j.
\]
Hence,  
\[
|\langle y-\gamma(t),  e_j(t) \rangle| \lsim \eps_j
\]
for all $j$,  which shows that indeed $B_{\mathbf \eps}(s) \subset B_{C \mathbf \eps}(t)$ for a sufficiently big constant $C$. 

For the second assertion,  we note that by Lemma \ref{lem::frenetchange}, because the curve is non-degenerate,  
\[
|\langle \gamma(s)-\gamma(t),  e_k(t) \rangle| \sim |s-t|^k \quad \mathrm{and} \quad |\langle e_j(s),  e_k(t) \rangle| \sim |s-t|^{|k-j|} \quad \mathrm{for} \; j \neq k.
\]
Then,  for a point $y=\gamma(s)+\sum_{j=1}^n s_j e_j(s)$,
\[
\langle y-\gamma(t),  e_k(t) \rangle=\langle \gamma(s)-\gamma(t),  e_k(t) \rangle+ \sum_{j=1}^n s_j \langle e_j(s), e_k(t) \rangle.
\]
For a point to be contained in $B_{c\mathbf \eps}(s) \cap B_{c\mathbf \eps}(t)$,  the $s_j$ must be such that $|s_j| \leq c \eps_j$ and also 
\begin{equation} \label{eq::sjcond}
\left| \langle \gamma(s)-\gamma(t),  e_k(t) \rangle+ \sum_{j=1}^n s_j \langle e_j(s), e_k(t) \rangle \right| \leq c \eps_k.
\end{equation}

If the minimum in $\Delta$ is attained for one of the $\eps_k^{1/k}$ (so with $j=0$),  then 
\[
\left| \langle \gamma(s)-\gamma(t),  e_k(t) \rangle \right| \gsim \eps_k.
\]
By the calculation from the first part,  the rest of the terms are $O(c \eps_k)$ for $|s_j| \leq c \eps_j$.  Thus,  for $c$ sufficiently small compared to the implicit constants,  \eqref{eq::sjcond} will fail.  This shows the second assertion.  

For the final assertion,  we let $j<k$ be such that $\Delta=(\eps_j^{-1} \eps_k)^{1/(k-j)}$.  If $|s-t| \leq \Delta$,  there is nothing to show.  Any point in $B_{\mathbf \eps}(s) \cap B_{\mathbf \eps}(t)$ then must be such that $|s_j| \leq \eps_j$,  while simultaneously by \eqref{eq::sjcond},
\[
|s_j c+d|:=\left| s_j \langle e_j(s), e_k(t) \rangle+\langle \gamma(s)-\gamma(t),  e_k(t) \rangle+ \sum_{i \neq j} s_i \langle e_i(s), e_k(t) \rangle \right| \leq \eps_k,
\]
where $c$ and $d$ depend on $s, t$ and $s_i$ for $i \neq j$ but are independent of $s_j$.  We note that by Lemma \ref{lem::frenetchange},
\[
|c| \gsim |s-t|^{k-j} \gsim \left(\frac{|s-t|}{\Delta} \right)^{k-j} \Delta^{k-j}=\left(\frac{|s-t|}{\Delta} \right)^{k-j} \eps_j^{-1} \eps_k,
\]
which shows that for any choice of value of $s_i,  i \neq j$,  the second inequality is satisfied on an interval of length at most
\[
|c|^{-1} \eps_k \lsim \left(\frac{\Delta}{|s-t|} \right)^{k-j} \eps_j.
\]
Hence,  denoting the vector of the coordinates $s_i,  i \neq j$ by $\underline{s}$ and with $c, d$ and $y$ as above,
\begin{align*}
\vol(B_{\mathbf \eps}(s) \cap B_{\mathbf \eps}(t)) &=\int_{-\eps_1}^{\eps_1} \dots \int_{-\eps_n}^{\eps_n} d\underline{s} \; \int_{-\eps_j}^{\eps_j} ds_j \; \mathbf{1}_{B_{\mathbf \eps}(t)}(y) \\
&\leq \int_{-\eps_1}^{\eps_1} \dots \int_{-\eps_n}^{\eps_n} d\underline{s} \; \int_{-\eps_j}^{\eps_j} ds_j \; \mathbf{1}_{|s_j c+d|\leq \eps_k}  \\
&\leq \int_{-\eps_1}^{\eps_1} \dots \int_{-\eps_n}^{\eps_n} \; |c|^{-1} \eps_k \; d\underline{s} \\
&\lsim \left(\frac{\Delta}{|s-t|} \right)^{k-j} \eps_1 \dots \eps_n.
\end{align*}
\end{proof}
For the configuration $\eps_1, \dots, \eps_n$,  we define the dual configuration $(\widetilde{\eps_i})_{0 \leq i \leq n}$ by
\[
\widetilde{\eps_i}:=\eps_{n}/\eps_{n-i}
\]
and the dual step size by
\[
\widetilde{\Delta}:=\min \left\{ (\widetilde{\eps_j}^{-1} \widetilde{\eps_k})^{\frac{1}{k-j}},  0 \leq j<k \leq n \right\}.
\]
\begin{remark}
From the definition,  it is clear that $\Delta=\widetilde{\Delta}$.  
\end{remark}
Using the material in this chapter,  we get the following result. 
\begin{theorem} \label{thm::nearcurvedual}
Let $\gamma$ be a non-degenerate curve in $\R^n$,  let $\gamma^\ast$ be its dual curve and let $\eps_1 \geq \dots \geq \eps_n$.  Let $\Delta$ be the associated step size and $V(\eps_i)$ be the adapted neighbourhood.  
Let $\widetilde{V}$ be the $\widetilde{\eps_1} \times \dots \times \widetilde{\eps_i}$-adapted neighbourhood to the dual curve $\gamma^\ast$.  

Then,  the intersection size of hyperplanes $H$ with $V$ is governed by whether the associated point $x_H$ lies in the dual neighbourhood $\widetilde{V}$.  To be precise,  for any $H$ with $x_H \in \widetilde{V}$ we have that
\[
\vol(H \cap V) \gsim \eps_1 \dots \eps_{n-1},
\]
while if 
\[
\vol(H \cap V) \gsim \frac{1}{A} \eps_1 \dots \eps_{n-1},
\]
for some parameter $A$,  then $x_H \in \widetilde{V}_A:=\widetilde{V}((A \widetilde{\eps_i}))$.  Understanding hyperplanes of large intersection with $V$ of a certain form thus reduces to counting the associated points near the dual curve $\gamma^\ast$,  where `near' is understood in terms of the adapted neighbourhoods for the dual configuration $(\widetilde{\eps_j})_j$.  
\end{theorem}
\begin{proof}
We let $N:=\lfloor \Delta^{-1} \rfloor$ and $t_j:=\frac{j}{N}$.  By Lemma \ref{lem::minimalcover},  any hyperplane of significant intersection has to have large intersection with one of the boxes $B_{\mathbf \eps}(t_j)$.  By Lemma \ref{lem::boxdual},  this exactly corresponds to having a point in the dual box $\widetilde{B_{\mathbf \eps}(t_j)}$ with the dual,  with the same behaviour as in the conclusion of the theorem.  By definition of the dual box,  it is adapted to the flag dual to $V_0=\{0\} \subset V_1 \subset \dots \subset V_n=\R^n$ with parameters $\widetilde{\eps_i}=\frac{\eps_n}{\eps_{n-i}}$ and is centered at the point corresponding $x_{H_B}$ to $\gamma(t_j)+V_{n-1}$.  By Lemma \ref{lem::projdualcurve},  the flag dual to the osculating flag of $\gamma$ is the osculating flag of the dual curve $\gamma^\ast$ and the $\gamma(t_j)+V_{n-1}$ affine subspace exactly corresponds to $\gamma^\ast(t_j)$.  Hence,  the dual box indeed corresponds to the $\widetilde{\eps_1} \times \dots \times \widetilde{\eps_n}$ adapted neighbourhood of the dual curve $\gamma^\ast$.  
\end{proof}

We can now restate and prove Theorem \ref{thm::mainduality} from the introduction.
\begin{restated}{Theorem \ref{thm::mainduality}}
There is a finite collection of matrices $\Phi_k \in SL_n(\Z)$ with orthogonal columns such that the following holds. Let $\gamma$ be a non-degenerate curve in $\R^n$ with dual curve $\gamma^\ast$,  let $c_0 \geq \eps_1 \geq \dots \geq \eps_n>0$ and let $Q \in \N$ be such that
\[
\eps_1 \cdot \dots \cdot \eps_n \geq C Q^{-(n+1)},
\]
where $c_0$ and $C$ are the constants from Theorem \ref{thm::maindualitybox}.  Let $\Delta$ be the associated step size,  $N:=\lfloor \Delta^{-1} \rfloor$,  and let $t_1,  \dots,  t_N$ be as in Lemma \ref{lem::minimalcover}.  Let $V_j:=B_{\mathbf \eps}(t_j)$,  $1 \leq j \leq N$,  be the boxes of the minimal cover of $V(\mathbf \eps)$,  and let $\widetilde{V}_j$ be the $\widetilde{\eps_1} \times \dots \times \widetilde{\eps_n}$-box adapted to the dual curve $\gamma^\ast$ at time $t_j$,  where $\widetilde{\eps_i}=\eps_n/\eps_{n-i}$.  For $B \geq 1$,  let $(\widetilde{V}_j)_B$ denote the corresponding box with parameters $B\widetilde{\eps_1},  \dots,  B\widetilde{\eps_n}$.  Set $\widetilde{Q}:=\frac{1}{\eps_n Q}$.  

For any parameter $A \geq 1$,  if
\[
\mathcal{N}(Q,  V_j) \geq C A \eps_1 \dots \eps_n Q^{n+1}
\]
for some $j$, then there is a $\Phi_k$ in the collection above such that for all $1 \leq i \leq r<n$ there are $B_i \geq 1$ and $m_i \leq \widetilde{Q}/B_i$ such that
\[
\prod_{i=1}^r \frac{\widetilde{Q}}{B_i m_i} \gsim A,
\]
and there are rational points $\vec{a_i}/m_i \in (\widetilde{\Phi_k V_j})_{B_i} \cap K(C)$ with $\vec{a_i} \in \Z^n$.  The rational points $\vec{a_i}/m_i$ are in general position,  meaning that no $s+2$ of them are contained in an $s$-dimensional affine subspace.  The rational points in $V_j$ lie in the $(n-r)$-dimensional rational subspace $H_1 \cap \dots \cap H_r$,  where $H_i$ is the hyperplane corresponding to $\vec{a_i}/m_i$ under projective duality (in the coordinates given by $\Phi_k$).  Here, the $\widetilde{\Phi_k V_j}$ are the boxes of the adapted cover of the dual curve $(\Phi_k \gamma)^\ast$.
\end{restated}
\begin{proof}[Proof of Theorem \ref{thm::mainduality}]
Follows immediately from Theorem \ref{thm::maindualitybox} applied to the boxes $V_j$ of the minimal cover,  since by the proof of Theorem \ref{thm::nearcurvedual} the dual box of $V_j$ is $\widetilde{V}_j$.  
\end{proof}

For general manifolds, we can use a similar principle to induce duality: If $\mathcal{M}=\{(x, f(x)): x \in U\}$, then partition $U$ into cubes of side length $\sqrt{\eps}$, linearise, and apply Theorem \ref{thm::maindualitybox} to each box individually. If $f$ has curvature along a row (which can be arranged under assumptions as weak as $f$ having no planar point), then under dualisation we still get a sparse problem along the row, to which we can try to apply the determinant method.
\section{Geometric Lemmata} \label{sec::geomlemmata}
We use the osculating flag and the Frenet frame of a non-degenerate curve as introduced in Section \ref{subsec::basics}.  The following lemma is a standard result about the ways in which the Frenet frame changes.
\begin{lemma} \label{lem::frenetchange}
Let $\gamma$ be a non-degenerate curve in $\R^n$ with osculating flag $V_1 \subset \dots \subset V_n$ and Frenet-Serret frame $e_1(t), \dots, e_n(t)$.  Then,  for any $s, t \in [0,1]$, if we write
\[
\gamma(t)-\gamma(s)=\sum_{i=1}^n e_i(t) c_i
\]
and
\[
e_j(t)-e_j(s)=\sum_{i=1}^n e_i(t) c_{ij},
\]
we have that
\[
|c_j| \sim |s-t|^j \quad, \quad |c_{jj}| \sim |s-t|^2 \quad  \mathrm{and} \quad |c_{ij}| \sim |s-t|^{|i-j|} \; \mathrm{for} \; i \neq j,
\]
where the implicit constants depend only on the curvature bounds for $\gamma$. 
\end{lemma} 
\begin{proof}
This follows easily from the Frenet--Serret equations (see e.g.\ \cite[Chapter 2]{kuehnel}),  which assert that for each $j$,  $\frac{d}{dt} e_j(t)$ can be smoothly written (with coefficients bounded away from zero for a non-degenerate curve) in terms of $e_{j-1}$ and $e_{j+1}$,  and Taylor approximation.
\end{proof}

\begin{lemma} \label{lem::Rnvolnghbhd}
Let $\gamma \in C^n([0,1]; \R^n)$ be a non-degenerate curve with $|\det(\gamma'(t), \dots, \gamma^{(n)}(t))| \geq c$ and $|\gamma^{(m)}(t)| \leq c^{-1}$ for all $t \in [0,1]$ and $1 \leq m \leq n$, and let $e_1(t), \dots, e_n(t)$ be its Frenet frame. Let $\lambda_1 \geq \lambda_2 \geq \dots \geq \lambda_n > 0$ and for $s \in [0,1]$ set
\[
V_s:=\left\{\gamma(s)+\sum_{i=1}^n \beta_i e_i(s): |\beta_i| \leq \lambda_i \textrm{ for all } 1 \leq i \leq n \right\}.
\]
Let $I \subset [0,1]$ be an interval of length $|I|=t \leq 1$, let $t_0 \in I$ and let $V_I:=\bigcup_{s \in I} V_s$. For a vector $y \in \R^n$,  let $|\langle V_I,  y \rangle|$ denote the diameter of the set $\{\langle x, y \rangle: x \in V_I\}$. Then the following hold, with implicit constants depending only on $n$ and $c$.
\begin{enumerate}
\item For every $1 \leq i \leq n$,
\[
|\langle V_I,  e_i(t_0) \rangle| \lsim w_i:=t^i+\sum_{j=1}^i \lambda_j t^{i-j}.
\]
\item Set $M_0:=1$ and $M_i:=\max(t M_{i-1},  \lambda_i)$ for $1 \leq i \leq n$. Let $i_1<\dots<i_k$ be the indices $i \geq 1$ with $M_i=\lambda_i$, and set $i_{k+1}:=n+1$, $m_j:=i_{j+1}-i_j$ (and $i_1:=n+1$ if $k=0$). Then
\[
\vol(V_I) \lsim \prod_{i=1}^n w_i \sim t^{i_1(i_1-1)/2} \prod_{j=1}^k \lambda_{i_j}^{m_j} t^{m_j(m_j-1)/2}.
\]
In particular,
\[
\vol(V_I) \lsim t^{n(n+1)/2}+\sum_{\substack{1 \leq i_1<\dots<i_k \leq n \\ 1 \leq k \leq n}} t^{i_1(i_1-1)/2} \prod_{j=1}^k \lambda_{i_j}^{m_j} t^{m_j(m_j-1)/2}.
\]
\end{enumerate}
As only upper bounds on the $|\beta_i|$ are used, the same bounds hold for neighbourhoods defined by $\beta_i \sim \lambda_i$ and for unions over finitely many points $t_j \in I$.
\end{lemma}
\begin{proof}
Throughout, write $V_m(u):=\Span(\gamma'(u), \dots, \gamma^{(m)}(u))=\Span(e_1(u), \dots, e_m(u))$.

\textit{Step 1: Frenet coefficients.} We first show that for all $s, u \in [0,1]$ and $1 \leq i, k \leq n$,
\begin{equation} \label{eq::Rnvol-gamma}
|\langle \gamma(s)-\gamma(u),  e_i(u) \rangle| \lsim |s-u|^i
\end{equation}
and
\begin{equation} \label{eq::Rnvol-frame}
|\langle e_k(s),  e_i(u) \rangle| \lsim |s-u|^{|k-i|}.
\end{equation}
For \eqref{eq::Rnvol-gamma}, Taylor's theorem gives
\[
\gamma(s)-\gamma(u)=\sum_{r=1}^{i-1} \frac{(s-u)^r}{r!} \gamma^{(r)}(u)+O(|s-u|^i),
\]
and the sum lies in $V_{i-1}(u)$, which is orthogonal to $e_i(u)$. For \eqref{eq::Rnvol-frame} with $k<i$, the Gram--Schmidt procedure gives $e_k(s)=\sum_{m=1}^k a_{km}(s) \gamma^{(m)}(s)$ with coefficients bounded in terms of $n$ and $c$; indeed, by Hadamard's inequality, $|\det(\gamma', \dots, \gamma^{(n)})| \leq \vol_k(\gamma', \dots, \gamma^{(k)}) \prod_{m>k} |\gamma^{(m)}|$, so all Gram determinants of $\gamma', \dots, \gamma^{(k)}$ are bounded below. For $m \leq k<i$, expanding $\gamma^{(m)}(s)$ around $u$ to order $i-m-1$ gives a polynomial part in $V_{i-1}(u) \perp e_i(u)$ and a remainder of size $O(|s-u|^{i-m})$, which is $O(|s-u|^{i-k})$ because $m \leq k$ and $|s-u| \leq 1$. For $k>i$, we write $\langle e_k(s),  e_i(u) \rangle=\langle e_i(u),  e_k(s) \rangle$ and apply the previous case with the roles of $s$ and $u$ exchanged. The case $k=i$ is trivial. (Two-sided versions of these bounds are contained in Lemma \ref{lem::frenetchange}.)

\textit{Step 2: widths.} Let $x=\gamma(s)+\sum_{k=1}^n \beta_k e_k(s) \in V_s$ with $s \in I$, so that $|s-t_0| \leq t$. By \eqref{eq::Rnvol-gamma} and \eqref{eq::Rnvol-frame},
\[
|\langle x-\gamma(t_0),  e_i(t_0) \rangle| \lsim t^i+\sum_{k \leq i} \lambda_k t^{i-k}+\sum_{k>i} \lambda_k t^{k-i}.
\]
As $\lambda_k \leq \lambda_i$ for $k>i$ and $t \leq 1$, the last sum is at most $(n-i) \lambda_i$ and can be absorbed into the term with $k=i$. Since the diameter is at most twice the supremum, this shows the first assertion.

\textit{Step 3: volume.} By the first assertion, $V_I$ is contained in the box $\gamma(t_0)+\{\sum_{i=1}^n c_i e_i(t_0): |c_i| \leq C w_i\}$, so $\vol(V_I) \leq (2C)^n \prod_{i=1}^n w_i$. Setting $\lambda_0:=1$, we have $w_i=\sum_{j=0}^i \lambda_j t^{i-j} \sim \max_{0 \leq j \leq i} \lambda_j t^{i-j}$, and because $\lambda_j t^{i-j}=t \cdot \lambda_j t^{i-1-j}$ for $j \leq i-1$, this maximum is exactly $M_i=\max(t M_{i-1},  \lambda_i)$. Consequently, $M_i=t^i$ for $i<i_1$ and $M_i=\lambda_{i_j} t^{i-i_j}$ for $i_j \leq i<i_{j+1}$, so
\[
\prod_{i=1}^n M_i=t^{1+\dots+(i_1-1)} \prod_{j=1}^k \lambda_{i_j}^{m_j} t^{0+1+\dots+(m_j-1)}=t^{i_1(i_1-1)/2} \prod_{j=1}^k \lambda_{i_j}^{m_j} t^{m_j(m_j-1)/2}.
\]
The final assertion follows by bounding this term by the sum over all increasing sequences $i_1<\dots<i_k$.
\end{proof}
\begin{lemma} \label{lem::distplanetocurve}
Let $\gamma: [0,1] \to \R^n$ be a non-degenerate curve. There is a constant $c_0$ depending only on the dimension and on the curvature bounds for $\gamma$ such that for any hyperplane $H$ intersecting the curve in a point $t_0$,  for any $s_1<s_2 \leq c_0$,  we have that the size of the intervals in which the curve is close to the hyperplane satisfies
\[
|\{t: |\langle \gamma(t)-\gamma(t_0),  \vec{n}_H \rangle| \leq s_1\}| \lsim \left(\frac{s_1}{s_2} \right)^{1/n} |\{t: |\langle \gamma(t)-\gamma(t_0),  \vec{n}_H \rangle| \leq s_2\}|.
\]
The implicit constant also depends only on the dimension and on the curvature bound for $\gamma$.  
\end{lemma}
\begin{proof}
Let $\alpha_j:=\frac{1}{j!}\langle \gamma^{(j)}(t_0),  \vec{n}_H \rangle$ and note that by Taylor approximation,  for $t$ sufficiently close to $t_0$,
\[
|\langle \gamma(t)-\gamma(t_0),  \vec{n}_H \rangle|=\left|\sum_{j=1}^n \alpha_j (t-t_0)^j \right|+O(|t-t_0|^{n+1}).
\]
In particular,  this implies that 
\[
|\langle \gamma(t)-\gamma(t_0),  \vec{n}_H \rangle| \sim \max_{j=1, \dots, n} |\alpha_j| |t-t_0|^j.
\]
Hence,  for sufficiently small $s$,  we get that if $t$ is a point such that $s=|\langle \gamma(t)-\gamma(t_0),  \vec{n}_H \rangle|$,  then
\[
|t-t_0| \sim \min_{j=1}^n( |\alpha_j|^{-1} s)^{1/j}.
\]
For the two intervals corresponding to $s_1$ and $s_2$,  we consequently find that 
\[
|\{t: |\langle \gamma(t)-\gamma(t_0),  \vec{n}_H \rangle| \leq s_i\}| \sim \min_{j=1}^n( \alpha^{-1}_j s_i)^{1/j}.
\]
Now,  if for $s_i=s_2$ the minimum on the right is attained for some index $j$,  then 
\begin{align*}
|\{t: |\langle \gamma(t)-\gamma(t_0),  \vec{n}_H \rangle| \leq s_1 \}| &\lsim ( \alpha^{-1}_j s_1)^{1/j}=\left(\frac{s_1}{s_2} \right)^{1/j} ( \alpha^{-1}_j s_2)^{1/j} \\
&=\left(\frac{s_1}{s_2} \right)^{1/j} \min_{j=1}^n( \alpha^{-1}_j s_2)^{1/j} \\
&\lsim \left(\frac{s_1}{s_2} \right)^{1/j} |\{t: |\langle \gamma(t)-\gamma(t_0),  \vec{n}_H \rangle| \leq s_2 \}|.
\end{align*}
As $j \leq n$, this shows the claim.  
\end{proof}
For sets $A, B$ in $\R^3$ we set
\[
\conv(A, B):=\{\lambda a+(1-\lambda)b: a \in A, b \in B, 0 \leq \lambda \leq 1\}.
\]
\begin{lemma} \label{lem::R3volgain}
Let $\gamma$ be a non-degenerate curve in $\R^3$. Let $\eta>0$ and let $V_j=V(t_j)$ be its $1 \times \eta \times \eta^2$-adapted minimal cover with coarseness $\Delta=\eta$.  Then for any $t_i \neq t_j$,  
\[
\vol(\conv(V_i, V_j)) \lsim |t_i-t_j|^2 \max\left(\eta,  |t_i-t_j|^2 \right).
\]
\end{lemma} 
\begin{proof}
Proving this is just a calculation that can be easily checked,  but because the geometric intuition is helpful for understanding later proofs,  we will give a rough sketch here.  Think of $\gamma(t)=(t, t^2, t^3)$ as the moment curve, $t_i=0$ and a point $t=t_j$.  Ignoring the change in base point for a moment and only considering the change in the orientation of the long side of the box from along $(1, 0, 0)$ to along $(1,  2t, 3t^2)$,  when we project along the first coordinate axis (or intersect with a plane of the form $\{1\} \times \R^2$,  say),  the projected curve and the projected neighbourhoods look like a non-degenerate planar curve akin to $(t, t^2)$ with its $\eta \times \eta^2$-minimal cover up up to a value of $t=c_0$ for a small $c_0$;  the first neighbourhood looks like a box of width $\eta$ in the second coordinate direction and $\eta^2$ in the third,  while the second box looks similar but being centered at a point like $(t, t^2)$ with the long axis being in direction $(1,  t)$.  The area contained between the two boxes is then $t \cdot \max(\eta^2,  \eta t)=\eta t^2$, where the first term comes from the width of the whole picture and the second from the height of the second box in the third coordinate (which is gained at an angle $t$ over an interval of length $\eta$).     This is the most important picture to keep in mind,  but does not account for the presence of $\max(\eta, t^2)$.  

All layers of the convex hull (intersection with planes of the form $\{\lambda\} \times \R^2$ for $1/2 \leq \lambda \leq 1$,  say) linearly scaled copies of this picture,  and this seems to argue for the total volume to be of order $1 \times \eta t^2$ (volume of the layers times volume in each).  However,  this is only because we have not taken the shift of the basepoint into account,  which translates everything by $t^3$ in the third coordinate direction; thus,  we not only have a height of $\eta t$ but an extra uncertainty of $t^3$ in the height in the third coordinate direction and because the contribution of the base point does not scale between layers while the previous picture does,  this leads to genuine volume gain overall.  Thus,  the volume of the convex hull is really $1 \times t \times \max(\eta t,  t^3)$,  which is the claimed result.  
\end{proof}

\section{Counting rational points near curves in $\R^3$} \label{sec::R3curves}
In this section,  we count the rational points near a non-degenerate curve $\gamma : [0, 1] \to \R^3$.
\begin{restated}{Theorem \ref{thm::strictcount}}
Let $\gamma : [0, 1] \to \R^3$ be a non-degenerate curve.  For all $\delta \geq Q^{-1/3}$,
\[
\mathcal{N}(\delta, Q) \lsim \delta^2 Q^2.
\]
\end{restated}
\begin{remark} \label{rem::R3sharp}
For $\delta \geq Q^{-1/3+\eta}$,  the contribution of the patches with an excess of rational points is $O(\delta^2 Q^{2-c \eta})$,  and arguing as in the planar case one expects a lower bound and an asymptotic formula in this range.  At $\delta=Q^{-1/3}$ itself,  a rational tangent line can contribute a positive proportion of the main term,  so only the upper bound holds there.  Similarly,  points coming from these examples start to dominate for $\delta < Q^{-1/3}$.  For example,  for $\gamma$ the moment curve $(t, t^2, t^3)$,  the rational points lying exactly on the first coordinate axis start to dominate the main term at this threshold.
\end{remark} 
The key part of the proof of Theorem \ref{thm::strictcount} is to transfer the problem of counting rational points near curves in the dense setting of $\delta \geq Q^{-3/5}$ to the dual counting problem of rational points in sparse (that is, $\delta \leq Q^{-3/5}$) neighbourhoods of the dual curve.  This dual problem can then be solved by combining geometric considerations with the determinant method.  

More precisely,  after applying the duality of Theorem \ref{thm::mainduality} we are left with two instances we have to control regarding to the sparse dual counting problem.  Firstly,  the proportion of $1 \times \sqrt{\frac{\delta}{Q}} \times \frac{\delta}{Q}$-neighbourhoods $\widetilde{V}_j$ which contain any $\delta^{-1}$ rational,  and then the proportion of neighbourhoods that contain two distinct rationals of respective height $m_1, m_2$ such that $m_1 m_2 \leq \delta^{-2}$.   

By the determinant method together with very basic volume considerations,  the first problem essentially reduces to considering when there is an interval $I$ and a rational hyperplane $H$ such that many of the neighbourhoods $\widetilde{V}_j$ making up $\widetilde{V}_I$ for some interval $I$ contain a $\delta^{-1}$-rational and all of those rationals lie on $H$.

One concern in the proof of Theorem \ref{thm::strictcount} are the neighbourhoods $V_j$ of the minimal cover in which all rationals lie on a single line and where correspondingly the dual neighbourhood $\widetilde{V}_j$ contains two distinct rationals of low denominator.  To control the instances of points containing a single rational,  we calculated the volume $V_I$ over an interval to get that they must lie in a single hyperplane,  but for the case of two distinct rationals,  this can (and has to) be refined.  We now use that the convex hull of the two neighbourhoods itself already has to contain significant volume in the case that there are four rationals not contained in a hyperplane.  

To bound this,  we use Lemma \ref{lem::R3volgain}.
\begin{prop} \label{prop::R3linecase}
Let $\gamma$ be a non-degenerate curve in $\R^3$,  let $M$ be an integer and  $\eps>0$.  Let $V_j=V(t_j),  j=1, \dots, N$ with $N \sim \eps^{-1}$ be the $1 \times \eps \times \eps^2$-adapted cover of the curve.  There is a set $J \subset \{1, \dots, N\}$ such that 
\[
|t_i-t_j| \gsim \min(M^{-1/2}, M^{-1} \eps^{-1/2})
\]
for all $i \neq j \in J$,  which in particular implies
\begin{align*}
|J| \lsim \max\left (\sqrt{M},  M \sqrt{\eps} \right),
\end{align*}
such that each $V_j,  j \in J$ contains two distinct rational points of height $m_1, m_2$ respectively such that $M/2 \leq m_1 m_2 \leq M$,  and 
it holds that for each $1 \leq i \leq N$,  if $V_i$ with $i \notin J$ contains two distinct rational points of height $m_1, m_2$ respectively such that $M/2 \leq m_1 m_2 \leq M$,  then the rational points in $V_i$ are such that they lie in a common hyperplane with the rational points of one of the $V_j, j \in J$.  
\end{prop}
\begin{proof}
It suffices to show that for each $1 \leq i \leq j \leq N$,  if $V_i$ and $V_j$ are such that both contain distinct rationals of height $m_{i, 1}, m_{i, 2}$ and $m_{j, 1}, m_{j, 2}$,  respectively,  with $M/2 \leq m_{i, 1} m_{i, 2},  m_{j, 1} m_{j, 2} \leq M$ and such that the four rationals are not contained in a common hyperplane,  then 
\[
|t_i - t_j| \gsim  \min(M^{-1/2}, M^{-1} \eps^{-1/2}).
\]
This follows from a quick application of the determinant method and the preceding lemma.  Indeed,  
\[
\frac{1}{M^2} \sim \frac{1}{m_{i, 1} m_{i, 2} m_{j, 1} m_{j, 2}} \lsim \vol(\conv(V_i,  V_j)) \lsim |t_i-t_j|^2 \max \left(|t_i-t_j|^2,  \eps \right),
\]
where the second inequality is due to the determinant method and the last one due to Lemma \ref{lem::R3volgain}.  With a case distinction for the maximum on the right, we see the claim.  
\end{proof}
We are now in the position to prove Theorem \ref{thm::strictcount}.
\begin{proof}[Proof of Theorem \ref{thm::strictcount}]
Let $\tau \leq 1/3$,  $\delta=Q^{-\tau}$ and consider the non-degenerate curve $\gamma$ and its $\sqrt{\frac{\delta}{Q}} \times \frac{\delta}{Q} \times \frac{\delta}{Q}$-adapted cover $V_j=V(t_j),  j=1, \dots, N$ with $N \sim \sqrt{\frac{Q}{\delta}}$.  By Theorem \ref{thm::mainduality} (with $B:=\max_i B_i$),  for every $j$ such that 
\[
N_j=\#(V_j \cap \mathcal{A}_Q) \gsim A Q^{n+1} \vol(V_j)
\]
the following holds:
\begin{enumerate}
\item There are $m_1,  m_2 \in \N$ and $B \geq 1$ such that 
\[
m_1 m_2 AB \lsim \delta^{-2}
\]
and there are two distinct rational points of respective height $m_1$ and $m_2$ in $(\widetilde{V}_j)_B$.  In this case,  the contribution of rationals in $V_j$ comes from a single dense line. 
\item There is an $m \in \N$,  $B \geq 1$ such that 
\[
mAB \lsim \delta^{-1}
\]
and there is a rational point of height $m$ lying in the $1 \times B \sqrt{\frac{\delta}{Q}} \times B \frac{\delta}{Q}$ neighbourhood $(\widetilde{V}_j)_B$ of $t_j$.  In this case,  the contribution from rationals in $V_j$ come from a single hyperplane. 
\end{enumerate}
We bound now both occurrences using the auxiliary results above.  

In the first case,  we use a dyadic decomposition to fix the size of $M \leq \delta^{-2}$ and $B \leq \delta^{-2}$.  Applying Proposition \ref{prop::R3linecase} with the values $M$ and $\eps=B \sqrt{\frac{\delta}{Q}}$ (enlarging the $B \frac{\delta}{Q}$ in the third coordinate to $B^2 \frac{\delta}{Q}$ in the process),  we see that we can partition the interval $[0,1]$ into subintervals $I_m$ of length $\min(M^{-1/2}, M^{-1} \eps^{-1/2})$ in which the rationals of any two of the lines lie in a common hyperplane.  

As all of the lines are pairwise coplanar,  we are left with two choices: either they all lie in the same plane,  or they all pass through the same rational point.

If they all lie in the same plane,  let $L_m$ be the common (necessarily rational) plane and let $V_{i_m}$ be the neighbourhood with the maximal contribution.  All of the rationals in the $V_i$ of the window are then contained in $L_m$,  and by the planar case of the Theorem,  as both of them have roughly the same number of points,  the hyperplane $L_m$ must contain so many rationals that over the entire intersection length of the hyperplane the contribution is at least that of the lines.  Hence,  in this case we can absorb the contribution of $V_i$ and $V_j$ into the bound from the hyperplane case.  This leaves the case in which all lines pass through the same point.  For those,  it is not hard to see that the total count of rationals is dominated by the maximum attained on a single line,  as the angle between the curve and the line through the point grows linearly as we move away from the point at which it is smallest.

In the case that there is only a single $V_j$ with rationals of height multiplying to roughly $M$,  this can be significantly more than the contribution from any hyperplane it is contained in.  However,  by Proposition \ref{prop::R3linecase},  this can happen in at most $\max\left (\sqrt{M},  M \sqrt{\eps} \right)$ many instances,  which we call the contribution from lines.  

For $M \geq \eps^{-1}=\frac{1}{B} \sqrt{\frac{Q}{\delta}}$,  the term on the right dominates and we get a total contribution of
\[
M \sqrt{\eps} \frac{\delta^{-2}}{BM} Q^4 \left(\frac{\delta}{Q} \right)^{5/2}=\frac{1}{\sqrt{B}} \left(\frac{\delta}{Q} \right)^{3/4} Q^2.
\]
Summing over the scales for $M$ and $B$ gives a total contribution of at most $\left(\frac{\delta}{Q} \right)^{3/4} Q^2 \log Q$,  which is less than $\delta^2 Q^2$ because $\tau<3/5$ (in particular for $\tau \leq 1/3$).  Hence the contribution of those points is negligible.

For all $M \leq \eps^{-1}$,  we get a total rational count proportional to $Q^{(3-\tau)/2}$,  which itself is coming from the $M=B=1$ case.  Indeed,  setting $\sum^\mathcal{D}$ to be the summation over a dyadic partition of the range and noting that $\sum^\mathcal{D}_k k^{-\alpha} \lsim 1$ for all $\alpha>0$,  we can estimate the total contribution of rationals for these lines by 
\[
\sum^\mathcal{D}_{1 \leq B \leq \delta^{-2}} \sum^\mathcal{D}_{1 \leq M \leq \eps^{-1}} \frac{\delta^{-2}}{MB} Q^4 \left(\frac{\delta}{Q} \right)^{5/2} \sqrt{M} \lsim \delta^{-2} Q^4 \left(\frac{\delta}{Q} \right)^{5/2}=Q^{(3-\tau)/2}.
\]

Since $\tau \leq 1/3$,  we have $Q^{(3-\tau)/2} \leq Q^{2-2\tau}=\delta^2 Q^2$,  so the lines contribute at most $\delta^2 Q^2$.  This is the only place where the restriction $\delta \geq Q^{-1/3}$ is needed,  and by Remark \ref{rem::R3sharp} it is sharp.

We are then left with estimating the contribution of the second kind,  in which the excess points on an interval $V_j$ with central point $t_j$ come from a single hyperplane $H=H_j$ that corresponds to a rational point $\vec{a}/m$ of height $m$ in physical space.  We let $(\lambda,  \theta,  z)$ be the coordinates of the point $\vec{a}/m-\widetilde{\gamma}(t_j)$ in the Frenet basis $\widetilde{e_1}(t_j),  \widetilde{e_2}(t_j),  \widetilde{e_3}(t_j)$.  In physical space,  this means that the hyperplane $H$ is at an angle $\theta$ to $\gamma^\prime(t_j)$ and at an angle $\lambda$ to $\gamma^{(2)}(t_j)$.  Now, use a dyadic decomposition to fix a value $\lambda \geq \left(\frac{\delta}{Q} \right)^{1/3}$; because when writing
\begin{equation} \label{eq::hyperchangeexp}
\langle \gamma(t)-\gamma(t_j),  \vec{n_H} \rangle=\alpha_0+\alpha_1 (t-t_j)+\alpha_2 (t-t_j)^2+\alpha_3 (t-t_j)^3+O((t-t_j)^4)
\end{equation}
with 
\[
|\alpha_0| \sim |z|,  \quad |\alpha_1| \sim \theta, \quad |\alpha_2| \sim \lambda,  \quad |\alpha_3| \sim 1,
\]
the maximal intersection any hyperplane with this value has with the $\frac{\delta}{Q}$ neighbourhood of $\gamma$ is then proportional to
\[
s_\lambda:=\sqrt{\lambda^{-1} \frac{\delta}{Q}},
\]
although depending on $\theta$ and $z$ this might also be shorter.  Values of $\lambda < \left(\frac{\delta}{Q} \right)^{1/3}$ do not have to be considered because the quadratic term is then dominated by the cubic term at the threshold of $\frac{\delta}{Q}$ in \eqref{eq::hyperchangeexp},  so they can be all absorbed into the $\lambda \sim \left(\frac{\delta}{Q} \right)^{1/3}$ case.  

Fix one of the dyadic scales and denote it by $\lambda_0$.  Now,  use a further dyadic decomposition on $1 \leq B \leq \left(\lambda_0 \frac{\delta}{Q} \right)^{-1/2}$,  setting $\theta=B \sqrt{\lambda_0 \frac{\delta}{Q}}$ (the value $\sqrt{\lambda_0 \frac{\delta}{Q}}$ is chosen such that for smaller $\theta$, the quadratic contribution dominates the linear at the threshold of $\frac{\delta}{Q}$,  meaning that they can be absorbed into the $B=1$ case) and fix a value of $B=B_0$ and $\theta_0$.  The longest possible intersection range is then 
\[
s_{\lambda_0, B_0}=\frac{1}{B_0} s_{\lambda_0}
\]
and this intersection range is attained on the interval of length $s_{\lambda_0}$ around $t_j$ only if $|z| \lsim B_0 \frac{\delta}{Q}$.  Moreover,  this intersection length might also be attained in the surrounding interval of total length 
\[
t_{\lambda_0, \theta_0}:=B_0 \sqrt{\lambda_0^{-1} \frac{\delta}{Q}}
\]
before the angle between curve and line gets too steep for the intersection to have a length as long.  This motivates the definition of the step size $\Delta^\prime=t_{\lambda_0, \theta_0}$ and the choice of a partition $t_i$ with 
\[
|t_i-t_{i+1}|=t_{\lambda_0, \theta_0}
\]
(in a short interval $[t_0-c_0, t_0+c_0]$ in which the curvature bounds hold) and neighbourhoods of the form
\[
W_i:=\{x \in \R^3: x=\widetilde{\gamma}(t_i)+\sum_{k=1}^3 \widetilde{e_k}(t_i) \beta_k,  \beta_1 \sim \lambda_0,   \beta_2 \sim \theta_0,  |\beta_3|\leq B_0^2 \frac{\delta}{Q} \}
\]
which are chosen such that each point $x_H$ at distance $\lambda_0, \theta_0$ to the dual curve $\widetilde{\gamma}$ (corresponding to a hyperplane $H$ at angle $\lambda_0, \theta_0$ to the curve in physical space) can lie in at most $O(1)$ possible neighbourhoods $W_i$ before the distance to the dual curve becomes too large.  

Now,  also fix an integer $m_0 \leq \frac{\delta^{-1}}{B_0}$ to be the height of the hyperplane $H$ (together with the $\lambda_0$ and $\theta_0$ already fixed) and note that as there is a contribution from at most
\[
\frac{\delta^{-1}}{B_0 m_0} \vol(V_j) Q^4
\]
many $Q$-rationals in each neighbourhood $V_i$ in which such a hyperplane intersects the neighbourhood of the curve in physical space and because the $m_0$-rationals in each neighbourhood $W_i$ can lead to at most $\frac{t_{\lambda_0, \theta_0}}{\sqrt{\frac{\delta}{Q}}}$ many $V_j$ to intersect such a hyperplane (with the prescribed height and the prescribed angle), 
the total contribution of rational points from such hyperplanes is at most 
\begin{align*}
&\#\{i: W_i \textrm{\; contains \; a \;} m_0 \textrm{-rational}\} \frac{t_{\lambda_0, \theta_0}}{\sqrt{\frac{\delta}{Q}}} \frac{\delta^{-1}}{B_0 m_0} \vol(V_j) Q^4 \\
&\sim \#\{i: W_i \textrm{\; contains \; a \;} m_0 \textrm{-rational}\}  \sqrt{\lambda_0^{-1} \frac{\delta}{Q}}  \frac{\delta^{-1}}{m_0}  \delta^2 Q^2.
\end{align*}
To bound the number of $W_i$ that can contain an $m_0$-rational we first apply Lemma \ref{lem::Rnvolnghbhd} and the determinant method  (Lemma \ref{lem::ratrepellancy}) to find that the only problem occurrences are hyperplanes $L$ in dual space dense in $m_0$-rationals and intersecting many $W_i$ and then use the planar case of the theorem to control the number of times this can occur. 

Firstly,  fix an interval $I$ of length $|I|=t \geq t_{\lambda_0, \theta_0}$ and note that by Lemma \ref{lem::Rnvolnghbhd} and because $\theta_0 t \geq B_0^2 \delta/Q$,  the volume of $W_I:=\bigcup_{i: t_i \in I} W_i$ is at most
\[
\vol(W_I) \lsim t^6+\lambda_0^3 t^3+\max(\lambda_0, t) \theta_0^2 t.
\]
By Lemma \ref{lem::ratrepellancy},  we must thus have
\begin{equation} \label{eq::R3mintsatisfying}
m_0^4 \lsim t^6+\lambda_0^3 t^3+\max(\lambda_0, t) \theta_0^2 t
\end{equation}
unless all $m_0$-rationals inside $W_I$ lie on a single hyperplane.  To firstly bound the `generic' contribution,  we need to show that if we let $t_{min}$ be the smallest $t$ satisfying \eqref{eq::R3mintsatisfying},  then 
\begin{equation} \label{eq::R3genericbound}
\frac{1}{t_{min}} \sqrt{\lambda_0^{-1} \frac{\delta}{Q}}  \frac{\delta^{-1}}{m_0} \lsim 1.
\end{equation}
To see this,  we note that it suffices to consider the case in which one of the first two terms attains the minimum on the right hand side of \eqref{eq::R3mintsatisfying} because making $B_0$ larger only increases the left hand side of \eqref{eq::R3genericbound}.  Solving for $t$ we then find
\[
t_{min} \gsim \min\left(m_0^{-2/3}, \lambda_0^{-1} m_0^{-4/3} \right)
\]
and thus a total contribution of
\begin{align*}
\frac{1}{t_{min}} \sqrt{\lambda_0^{-1} \frac{\delta}{Q}}  \frac{\delta^{-1}}{m_0} \delta^2 Q^2 &\lsim \max\left(m_0^{2/3}, \lambda_0 m_0^{4/3} \right) \sqrt{\lambda_0^{-1} \frac{\delta}{Q}} \frac{\delta^{-1}}{m_0}  \delta^2 Q^2\\
&\lsim Q^{(5\tau-3)/6} \max \left(\lambda_0^{-1/2} \delta^{1/3} m_0^{-1/3},  1 \right)  \delta^2 Q^2,
\end{align*}
where we used $\lambda_0 \leq 1$,  $m_0 \lsim \delta^{-1}$ and $\sqrt{\frac{\delta}{Q}} \delta^{-4/3}=Q^{(5\tau-3)/6}$ to bound the terms in the second part of the maximum.  We cannot bound the first term in the maximum any further,  but this is only because the result is sharp: For example,  for $\lambda_0= \left(\frac{\delta}{Q} \right)^{1/3}$, any single point of denominator $m_0$ contributes 
\[
\left(\frac{\delta}{Q} \right)^{1/3} \delta^{-1} m_0^{-1}=Q^{(2\tau-1)/3} m_0^{-1}=Q^{(5\tau-3)/6} \lambda_0^{-1/2} \delta^{1/3} m_0^{-1}
\]
many rational points in physical space,  and for example the moment curve contains $m_0^{2/3}$ many rational points of denominator $m_0$ exactly on the curve in dual space.   For the total count this term is harmless: as $\lambda_0 \geq \left(\frac{\delta}{Q} \right)^{1/3}$ and $m_0 \geq 1$,  we have $\lambda_0^{-1/2} \delta^{1/3} m_0^{-1/3} \leq Q^{(1-\tau)/6}$,  so the generic contribution is at most $Q^{(2\tau-1)/3} \delta^2 Q^2$ for each choice of the dyadic parameters $\lambda_0,  B_0,  m_0$.  Summing over these,  the generic contribution is $\lsim (\log Q)^3 Q^{(2\tau-1)/3} \delta^2 Q^2 \lsim \delta^2 Q^2$ for $\tau \leq 1/3$.

We now proceed to the case in which there is a hyperplane $L$ all of the $m_0$-rationals near the dual curve lie on.  In this case,  all the planes $H_j$ corresponding to the neighbourhoods $W_i$ that $L$ meets pass through the rational point $y:=L^\ast$ in physical space.  We first treat the case in which $y$ is close neither to a tangent line or an osculating plane of $\gamma$,  nor to the curve itself.  Letting $W_i$ (associated to a time $t_0:=t_i$) be one of the neighbourhoods the hyperplane intersects, define the parameter $t=t(t_0)$ to be the distance on the curve until there are three other neighbourhoods $W_j$ also containing an $m_0$-rational such that all  the $t_j$ associated to the $W_j$ are contained in $[t_0-t, t_0+t]$.  As a minor detail,  we first discount the case of $t>\lambda_0$ by noting that the neighbourhoods are contained within disjoint neighbourhoods of the $t \times t^2 \times t^3$-minimal cover of the dual curve then,  meaning that they cannot all lie on the same hyperplane $L$ and we are in the generic case,  which we have already dealt with.  We can thus assume $t \leq \lambda_0$.  

The hyperplane $L$ must be rational.  We let $k$ be its height and $\alpha_1$ and $\alpha_2$ be its first two coordinates in physical space,  or equivalently,  the angle of the hyperplane to $\widetilde{e_1}(t_0)$ and $\widetilde{e_2}(t_0)$ respectively.  By the determinant method (Lemma \ref{lem::ratrepellancy}),  any three $m_0$-rationals not lying on a line contain an area of at least $k/m_0^3$ with respect to the area measure of $L$,  and so,  if we project them along $e_1(t_0)$ (denoting the projection by $\pi_1$),  they contain an area of at least $\alpha_1 k/m_0^3$.  

Now,  we split into two cases,  firstly the one in which the density of the points is high enough that we expect every $W_i$ to contain a rational point,  and then the case in which this is not so. For the first case,  we note that if we have
\[
\frac{\alpha_1 k}{m_0^3} \lsim \vol(\pi_1(W_i)) \sim B_0^3 \lambda_0^{1/2} \left(\frac{\delta}{Q} \right)^{3/2},
\]
and thus 
\[
\alpha_1 \lsim \frac{m_0^3 B_0^3 \lambda_0^{1/2}}{k} \lambda_0^{1/2} \left(\frac{\delta}{Q} \right)^{3/2}.
\]

If we let 
\[
E_0:=\left\{x \in \R^3: x=\widetilde{\gamma}(t_0)+\sum_{k=1}^3 \widetilde{e_k}(t_0) \beta_k,  \beta_1 \sim \lambda_0 \right\},
\] 
then $L \cap E_0$ consists of a strip of width of at most $\lambda_0 \alpha_1$.  By the determinant method (Remark \ref{rem::rathyperorthrepellancy}),  we know that any $m_0$-rational not lying on $L$ must have a distance of at least $1/m_0 k$ to the plane.  If we let $u_1$ be the length of the interval around $t_0$ in at which the neighbourhoods $W_i$ have gained a height of at most $\lambda_0 \alpha_1$ over the hyperplane and $u_2$ be the length of the interval around $t_0$ until the height of $1/m_0 k$ over the hyperplane is attained,  then we can have at most a relative proportion of $u_1/u_2$ of the $W_i$ contain a rational.  By Lemma \ref{lem::distplanetocurve},  we have that
\[
\frac{u_1}{u_2} \lsim \left(\lambda_0 \alpha_1 m_0 k \right)^{1/3} \lsim m_0^{4/3}B_0 \lambda_0^{1/2} \left(\frac{\delta}{Q} \right)^{1/2} .
\] 
In each interval of length $I \geq \sqrt{\frac{\delta}{Q}}$ in which all neighbourhoods in dual space contain a rational of denominator $m_0$ and angle $B_0$ there is a maximum of at most 
\[
|I| \frac{\delta^{-1}}{m_0 B_0} \delta^2 Q^2
\]
many rational points near the curve in physical space 
coming from them (because this is true for the neighbourhoods making up the minimal cover).  Hence,  by the bound above, the relative density of $Q$-rationals in physical space on the interval of length $u_2$ is at most
\[
\frac{u_1}{u_2} \frac{\delta^{-1}}{B_0 m_0}  \delta^2 Q^2 \lsim m_0^{1/3} \lambda_0^{1/2} \delta^{-1} \left(\frac{\delta}{Q} \right)^{1/2} \delta^2 Q^2 \lsim Q^{(5\tau-3)/6} \delta^2 Q^2,
\]
where we used that $m_0 \lsim \delta^{-1}$ and 
\[
\delta^{-4/3} \left(\frac{\delta}{Q} \right)^{1/2}=Q^{(5\tau-3)/6} \leq 1.
\]
Hence,  the contribution coming from such hyperplanes $L$ remains bounded.  

For the second case,  assume that the hyperplane $L$ is such that 
\[
\frac{\alpha_1 k}{m_0^3} \geq \vol(\pi_1(W_i)).
\]
Let $s>0$ be the number such that
\[
\lambda_0^2 s^3=\alpha_1 \frac{k}{m_0^3};
\]
this $s$ is such that for an interval $I$ containing $t_0$ of length $|I|=s$,
\[
\pi_1(W_I) \sim \alpha_1 \frac{k}{m_0^3},
\]
where $W_I=\bigcup_{i, t_i \in I} W_i$.  By the planar case applied to the neighbourhoods $W_i$ around the curve $\gamma_0$,  there can be at most $O(1)$ many neighbourhoods $W_i$ in any intervals of length $s$ containing an $m_0$-rational on $L$. As each $W_i$ containing a rational leads to at most 
\[
\sqrt{\lambda_0^{-1} \frac{\delta}{Q}}  \frac{\delta^{-1}}{m_0}  \delta^2 Q^2
\]
many $Q$-rationals in physical space by the above,  the average density of points in the interval in which the hyperplane $L$ intersects the neighbourhood we consider near the dual curve is at most
\[
\frac{1}{s} \sqrt{\lambda_0^{-1} \frac{\delta}{Q}}  \frac{\delta^{-1}}{m_0}  \delta^2 Q^2.
\]
Moreover,  as in the previous case,  because of the repelling property of the hyperplane $L$,  around the interval $J$ of length $u_1$ in which the neighbourhoods intersect the hyperplane (that is, the height of them over the curve is at most $\lambda_0 \alpha_1$),  there is an interval of length $u_2$ up to height $1/m_0 k$ over the hyperplane in which there can be no $m_0$-rationals in dual space contained in the neighbourhoods at all,  and by Lemma \ref{lem::distplanetocurve},
\[
\frac{u_1}{u_2} \lsim \left(\lambda_0 \alpha_1 m_0 k \right)^{1/3}=\lambda_0 s m_0^{4/3}.
\]
Hence,  the average density on the interval of length $u_2$ around $t_0$ coming from the rationals on $L$ is at most
\[
\frac{u_1}{u_2} \frac{1}{s} \sqrt{\lambda_0^{-1} \frac{\delta}{Q}}  \frac{\delta^{-1}}{m_0}  \delta^2 Q^2 \lsim m_0^{1/3} \lambda_0^{1/2} \sqrt{\frac{\delta}{Q}} \delta^{-1} \delta^2 Q^2 \lsim Q^{(5\tau-3)/6} \delta^2 Q^2
\]
as above,  showing that they are again under control.  Hence,  once again,  having the $m_0$-rationals on hyperplane cannot significantly increase the number of rationals we need to deal with. 

The problem cases in which $y$ is close to a tangent line,  an osculating plane or to $\gamma$ itself can be treated similarly.  For example,  if $y$ is close to the osculating plane of $\gamma$ at $t_0$,  then the curve inside $L$ has a cusp at the point corresponding to $t_0$.  We can then partition the neighbourhoods around $t_0$ according to the curvature the intersected curve has and apply the planar case there.  The case with $y$ close to the tangent line can be treated similarly,  and for the case with $y$ close to the curve itself,  we can use duality once more.

We can now put everything together.  Firstly,  the patches $V_j$ with $N_j \leq A Q^4 \vol(V_j)$ contribute at most
\[
A \sum_j Q^4 \vol(V_j) \lsim A \sqrt{\frac{Q}{\delta}} \, Q^4 \left(\frac{\delta}{Q}\right)^{5/2} = A \delta^2 Q^2.
\]
For the remaining patches,  we have shown that the patches whose rationals lie on a line contribute at most $Q^{(3-\tau)/2} \leq \delta^2 Q^2$,  while the patches whose rationals lie on a plane contribute at most $Q^{(2\tau-1)/3}\delta^2Q^2$ in the generic case or $Q^{(5\tau-3)/6}\delta^2Q^2$ in the coplanar case for each choice of the dyadic parameters.  As there are only $O((\log Q)^5)$ such choices and both exponents are negative for $\tau \leq 1/3$,  the plane patches contribute at most $o(\delta^2 Q^2)$.  Altogether,  $\mathcal{N}(\delta, Q) \lsim \delta^2 Q^2$,  which concludes the proof.
\end{proof}

\printbibliography
\end{document}